\documentclass[12pt,leqno]{amsart}
\usepackage[all]{xy}
\usepackage[usenames,dvipsnames]{pstricks}
\usepackage[backgroundcolor=green,shadow,colorinlistoftodos]{todonotes}
\usepackage{amssymb,mathrsfs,euscript}
\usepackage{mathrsfs,euscript,amssymb,amsmath,bbold}
\usepackage[hypertex]{hyperref}
\textheight9in  \def\DATE{April 15, 2014}
\makeatletter
\providecommand\@dotsep{5}
\def\listtodoname{List of Todos}
\def\listoftodos{\@starttoc{tdo}\listtodoname}
\makeatother

\newtheorem{theorem}{Theorem}

\newtheorem*{Principle}{Principle}

\newtheorem{lemma}[theorem]{Lemma}
\newtheorem{proposition}[theorem]{Proposition}

\theoremstyle{definition}
\newtheorem{example}[theorem]{Example}
\newtheorem{remark}[theorem]{Remark}
\newtheorem{definition}[theorem]{Definition}

\def\Triang{\underline{{\mathcal T}ri}}
\def\Arc{\underline{{\mathcal A}rc}}
\def\Wind{\underline{{\mathcal S}ur}}
\def\Mult{{\rm coProd}}
\def\catC{{\EuScript C}}
\def\and{\ \hbox { and } \ }
\def\CC{{\mathfrak C}}
\def\Lex{{\rm Lex}}
\def\nsF{{\underline{F}}}
\def\({(\hskip -.15em(}    \def\){)\hskip -.15em)}
\def\bl{\big(\hskip -.2em\big(}    \def\br{\big)\hskip -.2em\big)}
\def\Des{{\rm Des}}

\def\card{{\rm card}}

\def\Span{{\rm Span}}
\def\nsm{\ns\ modular}

\catcode`\@=11
\def\nsBox{{\underline{\Box}}}
\def\Iso{{\tt Iso}}
\def\sqplus{{\hskip -.1em + \hskip .1em}}
\def\@evenfoot{\rule{0pt}{20pt}[\DATE] \hfill [{\tt \jobname.tex}]}
\def\@oddfoot{\rule{0pt}{20pt}{[\tt \jobname.tex}]\hfill [\DATE]}
\catcode`\@=13

\def\Gammacat#1#2{{\fatGamma}\(#2;#1\)}

\def\fatGamma{{\mathbf \Gamma}}

\def\Flag{{\it Flag}}\def\edge{{\it edge}}\def\leg{{\it Leg}}

\def\colim#1{\mathop{{\rm colim}}%
             \limits_{\rule{0em}{1em}\mbox{\scriptsize $#1$}}}
\def\nsP{{\underline{\oP}}}
\def\Cyc{{\tt Cyc}}
\def\Free{{\underline{\mathbb M}}}\def\Freesym{{{\mathbb M}}}
\def\sfModtriple{\Free}
\def\redukce#1{\vbox to .3em{\vss\hbox{#1}}}
\def\rediso{{\ \redukce{$\stackrel \cong\lra$}\ }}
\def\lra{{\longrightarrow}}

\def\Fin{{\tt Fin}}
\def\bfk{{\mathbb k}}
\def\sfLeg{{\sf Leg}}
\def\nsMod{{\underline{\rm Mod}}}
\def\nsModMod{{{\tt NsModMod}}}
\def\ModMod{{{\tt ModMod}}}
\def\cupred{{\hskip .1em \cup \hskip .10em}}
\def\sfD{{\sf D}}
\def\otred{{\hskip -.01em \otimes \hskip -.01em}}
\def\nsAss{\underline{{\mathcal A}{\it ss}}}
\def\Com{\mbox{{$\mathcal C$}\hskip -.3mm {\it om}}}
\def\Ass{{\mathcal A}{\it ss}}
\def\ttM{{\tt M}}\def\Set{{\tt Set}}\def\id{{\it id}}
\newcommand{\oP}{{\mathcal{P}}}
\newcommand{\ooo}[2]{\sideset{_{#1}}{_{#2}}{\mathop{\circ}}}
\def\stt#1{{\{#1\}}}\def\sfS{{\sf S}}
\def\bbN{{\mathbb N}}\def\ot{\otimes}
\def\k-Mod{\hbox{\tt $\bfk$-Mod}}
\def\MultCyc{{\tt MultCyc}}
\def\Sym{{\rm Sym}}\def\Mod{{\rm Mod}}
\def\NsCycOp{{\tt NsCycOp}}
\def\CycOp{{\tt CycOp}}\def\NsModOp{{\tt NsModOp}}\def\ModOp{{\tt ModOp}}
\def\Rada#1#2#3{#1_{#2},\dots,#1_{#3}}
\def\bbN{{\mathbb N}}
\def\Leg{{\it Leg\/}}
\def\Vert{{\it Vert\/}}
\def\ns{{non-$\Sigma$}}

\def\calP{{\mathcal P}}

\title{Modular envelopes, OSFT and nonsymmetric (non-$\Sigma$) modular operads}

\author{Martin Markl}

\thanks{The second author was supported by the Eduard \v Cech
  Institute P201/12/G028 and RVO: 67985840.}


\allowdisplaybreaks
\begin{document}

\allowdisplaybreaks
\bibliographystyle{plain}

\begin{abstract}
  Our {\bf aim\/} is to introduce and advocate \ns\ (non-symmetric)
  modular operads.  While ordinary modular operads were inspired by
  the structure of the moduli space of stable complex curves, \ns\
  modular operads model surfaces with open strings outputs.  An
  immediate {\bf application} of our theory is a short proof that the
  modular envelope of the associative operad is the linearization of
  the terminal operad in the category of \ns\ modular operads. This
  gives a succinct description of this object that plays an important
  r\^ole in open {\bf string field theory\/}. We also sketch further
  perspectives of the presented approach.
\end{abstract}

\maketitle

\vskip 1em

\begin{center}
Author's addresses:
\end{center}

\small
\begin{center}
{\sc Math.\ Inst.\ of the Academy, {\v Z}itn{\'a} 25,
         115 67 Prague 1, The Czech Republic}
\\ and \\
{\sc MFF UK, 186 75 Sokolovsk\'a 83, Prague 8, The Czech Republic}
\end{center}
\begin{center}
email:
{\tt  markl@math.cas.cz}
\end{center}

\tableofcontents

\section*{Introduction}

Operads have their \ns\ (non-symmetric) versions obtained by
forgetting the symmetric group actions. Likewise, for the \ns\ version
of cyclic operads one requires only the actions of the cyclic
subgroups of the symmetric groups, see e.g.\ definitions~II.1.4,
II.1.14 and Section~II.5.1 of \cite{markl-shnider-stasheff:book}. In
both cases we thus demand less structure. As we explain in
Example~\ref{dead_end} below, this straightforward approach fails for
modular operads whose \ns\ versions have been a mystery so~far.

There are, fortunately, some clues and inspirations, namely modular
envelopes of cyclic operads, introduced under the name modular
completions by the author in~\cite[Definition~2]{markl:la}. The
modular envelope $\Mod(*_C)$ of the terminal cyclic operad $*_C$ in
the category $\Set$ of sets turned out to be the terminal modular
operad in $\Set$, see~\cite[page~382]{markl:la}. Notice that the
linearization $\Span(*_C)$ of $*_C$ is the cyclic operad $\Com$ for
commutative associative algebras.

The modular envelope $\Mod(\underline*_C)$ of the terminal \ns\ cyclic
$\Set$-operad, described much later 
in~\cite{chuang-lazarev:dual-fey,doubek:modass}, turned out to be
surprisingly complex. (The authors
  of~\cite{chuang-lazarev:dual-fey,doubek:modass} worked with the
linearized version, i.e.\ with the cyclic operad $\Ass$ for associative
  algebras, but the linear structure is irrelevant here as everything
  important happens inside $\Set$.)  A derived version of this modular
  envelope was studied in \cite{Costello,Giansiracusa}.   Guided by
the example of  $\Mod(*_C)$, one would expect
$\Mod(\underline*_C)$ to be the (symmetrization of) the terminal
$\Set$-operad in the conjectural category of \ns\ modular operads.
The description of $\Mod(\underline*_C)$ given in~\cite[Theorem~3.1]{doubek:modass}
and its relation to the moduli space of Riemann surfaces with
boundaries~\cite[Theorem~3.7]{chuang-lazarev:dual-fey} 
therefore gives some feeling what \ns\ modular
operads should be.

Another natural requirement is that the conjectural category of \ns\
modular operads should fill the left bottom corner of the diagram
\[
\xymatrix@C = 3.5em@R = 3em
{\NsCycOp  \ar@/^.7em/[r]^{\Sym} \ar@/_/[d]_\nsMod  & \CycOp \ar@/^/[d]^{\Mod} \ar@/^.7em/[l]_\Des
\\
??   \ar@/_/[u]_\nsBox  \ar@/^.7em/[r]^{\Sym} &  \ModOp \ar@/^/[u]^\Box \ar@/^.7em/[l]_\Des
}
\]
in which $\Box: \ModOp \to \CycOp$ is the forgetful functor, $\Mod:
\CycOp \to \ModOp$ the modular
completion~\cite{markl:la} (known today as the modular envelope), $\Des:
\CycOp \to \NsCycOp$ the forgetful 
functor (the desymmetrization, not to be mistaken with
  Batanin's desymmetrization of~\cite{batanin:globular})
and the symmetrization
$\Sym: \NsCycOp \to \CycOp$ its left adjoint.

The category of (ordinary) modular operads contains the category of
cyclic operads as the full subcategory of operads concentrated in genus
$0$. Requiring the same from the category of \ns\ modular operads
leads to the notion of geometricity that does not have analog in the
symmetric world.

Our {\bf aim\/} is to introduce and advocate the notion of \ns\
(non-symmetric) modular operads. The main {\bf definitions} are
Definitions~\ref{opet_podleham} and~\ref{sec:un-biased-definition},
and the main {\bf result} is
isomorphism~(\ref{Jarka_stale_bez_prace-a}) of
Theorem~\ref{sec:modular-envelopes-6}. As an immediate {\bf
  application} we give a short,  elementary proof of the description
of the modular envelope $\Mod(\Ass)$ of the associative operad given
in~\cite{doubek:modass,chuang-lazarev:dual-fey}.

 \subsection*{Perspectives.} 
It turns out that the elements of the modular
envelope $\Mod(*_C)$ of the terminal cyclic $\Set$ operad $*_C$
describe isomorphism classes of oriented surfaces with
holes, and likewise the \ns\
modular envelope $\nsMod(\underline*_C)$ of the terminal \ns\ cyclic
$\Set$-operad describe isomorphism  classes of oriented surfaces with
teethed holes. These geometric objects describe interactions in closed
and open string field theory. Very crucially,  
Theorem~\ref{sec:modular-envelopes-6} asserts that both  $\Mod(*_C)$ and
$\nsMod(\underline*_C)$ are the {\bf terminal\/} objects in an appropriate
category of modular operads.

In Section \ref{lazarev} we consider the operad $*_D$ describing
associative algebras with an involution. It is a cyclic dihedral
operad in the sense of \cite[Section~3]{markl-remm:JA06} that equals
the  {M\"obiusisation\/} \cite[Definition
3.32]{braun} of the terminal cyclic operad $*_C$.  By
\cite[Theorem~3.10]{braun}, its modular envelope $\Mod(*_D)$ consist of
isomorphism classes of {\em non-oriented\/} surfaces with teethed
holes. We believe that there exists another version of modular
operads (say dihedral modular operads) such that $\Mod(*_D)$ is the
(symmetrization of the) 
{\bf terminal\/} modular operad of this~type.

We hope that similar reasoning applies to other objects, such
as the surfaces describing interactions in open-closed string field
theory that may include also D-branes, or even higher-dimensional
manifolds. In each case the corresponding type of modular operad
should reflect how a geometric object is composed from simper pieces,
e.g.\ pair of pants in the case of closed string field theory. 
We are lead to formulate

\begin{Principle}
For a large class of geometric objects there exists a version of
modular operads such that the set of isomorphism classes of these
objects is the terminal modular $\Set$-operad of a given type.
\end{Principle}

Several steps in this direction have already been made
in~\cite{Dyc-Kap} where various generalizations of dihedral,
quaternionic and other generalizations of cyclic structures based on
crossed simplicial groups of Krasauskas \cite{krasauskas:LMS87} and
Fiedorowicz-Loday \cite{fiedorowicz-loday:TAMS91} were studied, and their
relation to structured surfaces was clarified.

The nature of the above principle is similar to the 
cobordism hypothesis~\cite{freed} as they both describe objects of
geometric nature by purely categorial means, see again \cite{Dyc-Kap}
where the relation to the structured cobordism hypothesis in dimension
$2$ was explicitly formulated, cf.\ also Remark~4.3.7 of~\cite{Dyc-Kap}.

\subsection*{Notations and conventions} 
Throughout the paper, $\ttM = (\ttM,\ot,1)$ will stand for a complete and
cocomplete, possibly enriched, symmetric monoidal category, with the
initial object $0 \in \ttM$. Typical examples will be $\k-Mod$, the
category of modules over a commutative unital ring $\bfk$, or the cartesian
category $\Set$ of sets. By $\Span : \Set \to \k-Mod$ we denote the
$\bfk$-linear span, i.e.\ the left adjoint to the forgetful functor $\k-Mod
\to  \Set$.
The glossary of categories 
introduced and used throughout the paper is given in
Figure~\ref{ten_externi_disk_je_smejd.}, various forgetful functors
and their adjoints are listed in Figure~\ref{Jaruska_bez_prace}.

\begin{figure}
\[
\def\arraystretch{1.2}
\begin{array}{rl}
\Set,&\mbox {the cartesian monoidal category of sets,}
\cr
$\k-Mod$,&\mbox {the category of $\bfk$-modules,}
\cr
\Fin,&\mbox {the category of finite sets,}
\cr
\Cyc,&\mbox {the category of finite cyclically ordered sets,}
\cr
\MultCyc,&\mbox {the category of finite multicyclically ordered sets,}
\cr
\ModMod,&\mbox {the category of modular modules,}
\cr
\nsModMod,&\mbox {the category of \ns\ modular modules,}
\cr
\CycOp,& \mbox {the category of cyclic operads,}
\cr
\NsCycOp,& \mbox {the category of \ns\ cyclic operads,}
\cr
\ModOp,& \mbox {the category of modular operads,}
\cr
\NsModOp,& \mbox {the category of \ns\ modular operads,}
\cr
\Gammacat g{\sfS}, & \mbox {the category of \ns\ modular graphs.}
\end{array}
\]
\caption{\label{ten_externi_disk_je_smejd.}Notation of categories.}
\end{figure}

\begin{figure}
\[
\def\arraystretch{1.2}
\begin{array}{|c|c|}
\hline
\mbox {forgetful functor:}  & 
\mbox{\rule{20pt}{0pt}left adjoint:}\rule{20pt}{0pt}
\\
\hline\hline
\Box : \ModOp \longrightarrow \CycOp & \Mod : \CycOp \longrightarrow \ModOp
\\
\nsBox : \NsModOp \longrightarrow \NsCycOp & \nsMod : \NsCycOp \longrightarrow \NsModOp
\\
\Des : \CycOp  \longrightarrow \NsCycOp & \Sym:  \NsCycOp \longrightarrow  \CycOp
\\
\Des : \ModOp  \longrightarrow \NsModOp & \Sym:  \NsModOp \longrightarrow  \ModOp
\\
F: \ModOp \longrightarrow \ModMod  & \Freesym: \ModMod  \longrightarrow \ModOp
\\
\nsF: \NsModOp \longrightarrow \nsModMod  & \Free: \nsModMod  \longrightarrow \NsModOp
\\
U: \k-Mod \longrightarrow \Set  & \Span:\Set    \longrightarrow  \k-Mod
\\
\hline
\end{array}
\]
\caption{\label{Jaruska_bez_prace}
Forgetful functors and their left adjoints.}
\end{figure}

An order (without an adjective) in this paper will always be a linear
order of a finite set. By $*$ we denote a chosen one-point set.

\subsection*{Assumptions}
We assume certain familiarity with basic notions of
operad theory.  There exists rich and easily accessible literature, 
for instance the monograph
\cite{markl-shnider-stasheff:book}, overview articles
\cite{getzler:operads-revisited,markl:pokroky,markl:handbook} or a recent
account~\cite{loday-vallette}. For the reader's convenience, we 
recall a definition of
cyclic and modular operads based on finite sets in the Appendix.

\subsection*{Acknowledgment} 
We would like to express my thanks to Martin Doubek and Ralph Kaufmann
for inspiring suggestions and comments. Some of the ideas presented here
stemmed from our discussions during Ralph's visit of Prague in April
2014. We are also indebted to anonymous referees for helpful
comments that lead to substantial improvement of the exposition.

\section{Cyclic orders and the first try}
\label{Druhy_na_Vivat_tour!}

We open this section by recalling basic facts about cyclically ordered
sets and their morphisms.  Although we will actually 
need only combinations of {\em
  iso\/}morphisms of {\em totally\/} cyclically ordered finite sets,
see Remark~\ref{sec:multicyclic-orders-2}, 
we will present the definitions in full
generality, putting the material of this and the following sections
into a broader context. Our exposition follows
closely~\cite[Section~2.2]{Dyc-Kap}, cf.\ also the classical 
sources~\cite{huntington, novak}.  

\begin{definition}
A {\em partial cyclic order\/} on a set $C$ 
is a ternary relation $\CC \subset C^{\times 3}$
satisfying the following conditions.
\begin{itemize}
\item[(i)]  
The triplet $(c,c,c)$ belongs to $\CC$ for any $c\in C$.
\item[(ii)] 
If  $(a,b,c)\in\CC$ and $(b,a,c)\in\CC$, then
$\card\{a,b,c\}\leq 2$. 
 \item[(iii)] 
Let $a,b,c,d\in C$ be mutually  distinct elements.
 If $(a,b,c)\in\CC$ and $(a,c,d)\in\CC$, then $(a,b,d)\in\CC$
 and $(b,c,d)\in\CC$.
 \item[(iv)]
If $(a,b,c)\in\CC$, then $(b,c,a)\in\CC$. 
 \item[(v)]   
If $(a,b,c)\in\CC$, then $(a, a, c)\in\CC$.
 \end{itemize}
If $\CC$ satisfies, moreover, the condition:
\begin{itemize}
\item[(vi)] 
for any $a,b,c \in C$, either $(a,b,c)\in\CC$, or $(b,a,c)\in\CC$, 
\end{itemize}
then  $\CC$ is called a {\em total\/} cyclic order. A {\em cyclically ordered
  set\/} is a couple $C = (C,\CC)$ of a set with a~(partial or total) 
cyclic~order.
\end{definition}

\begin{example}
  Each set $C$ with less than two elements (including the empty one)
  has a unique total cyclic order $\CC = C^{\times 3}$. The disjoint
  union $S = C_1 \cup \cdots \cup C_b$ of cyclically ordered sets $C_i
  = (C_i,\CC_i)$, $1\leq i \leq b$, bears a natural induced cyclic
  order $\CC_1 \cup \cdots \cup \CC_b$,
  cf.~\cite[Definition~3.7]{novak}. Each subset $T \subset C$ of a
  cyclically ordered set $C = (C,\CC)$ bears an induced cyclic order
  $\CC|_T$.
\end{example}

To define morphisms, we need the following construction.
Assume that $\sigma : C' \to C''$ is a map of sets, that $C''$ bears a
partial cyclic order $\CC''$, and that one is given a partial linear order $\leq_c$
on each fiber $\sigma^{-1}(c)$, $c \in
C''$. 

\begin{definition}
The {\em lexicographic order\/}
$\CC' = \Lex\big(\sigma,\{\leq_c\}\big)$ on $C'$ is given  by postulating that 
$(a,b,c)\in\CC'$ if $\big(\sigma(a), \sigma(b),
\sigma(c)\big)\in\CC''$ and either
\begin{enumerate}
\item[(i)]  
The elements $\sigma(a), \sigma(b), \sigma(c)$  are mutually distinct, or

\item[(ii)] 
for some cyclic permutation $(x,y,z) \in \big\{(a,b,c),(b,c,a),(c,a,b)\big\}$ one has
\[
\sigma(x)=\sigma(y)\neq \sigma(z),  \and 
x\leq_{\sigma(x)} y,
\]  
\item[(iii)] 
or $\sigma(a)\!=\!\sigma(b)\!=\!\sigma(c)$ and 
$x\leq_{\sigma(x)} y\leq_{\sigma(x)}  z$  for some $(x,y,z) \in
\big\{(a,b,c),(b,c,a),(c,a,b)\big\}$. 
\end{enumerate}
\end{definition}

If $\CC'$ is a total cyclic order and each $\leq_c$ a total linear
order, then $\Lex\big(\sigma, \{\leq_c\}\big)$ is a total cyclic order.  If,
moreover, $C'$ and $C''$ are finite sets, a simple alternative
description of the lexicographic order is given
in~\cite[Example~2.2.2(4)]{Dyc-Kap}.

\begin{example}
  If $\sigma : C' \to C''$ is a monomorphism of sets and $C''$ is totally
  ordered, then the induced lexicographic order on $C'$ is the inverse
  image $(\sigma^{\times 3})^{-1}(\CC'')$ of
  $\CC''$ under the map $\sigma^{\times 3}:  C'^{\times 3} \to  C''^{\times 3}$.
The opposite extreme is when $A = (A,\leq)$ is a 
totally linearly ordered set and $\sigma : A \to *$ the (unique) map
to a one-point set.  
The total 
lexicographic order $ \CC$ induced on  $A$ by $\sigma$ is then given by
\begin{equation}
\label{eq:22}
(a,b,c)\in \CC \ \hbox  { if and only if }
a \leq b \leq c \ \hbox { or } b \leq c \leq a \ 
\hbox { or } c \leq a \leq b .
\end{equation}
\end{example}

The following notation will  often be used in the sequel.

\begin{definition}
\label{Kdy_prijede_Kveta_z_Rataj?} For a totally ordered set $A = (A,\leq)$ we
denote by $[A]$ the underlying  set of $A$ with the total cyclic
order~(\ref{eq:22}). We say that this cyclic order is {\em
  represented\/}  by the linear order $\leq$.
\end{definition}

We can finally define morphisms of cyclically ordered sets,
cf.~\cite[Definition~2.2.9]{Dyc-Kap}.

\begin{definition}
  Let $C' =(C',\CC')$ and $C'' =(C'', \CC'')$ be partially cyclically
  ordered sets.  

If the order $\CC'$ is  total, then a morphism $(C',\CC')\to
  (C'',\CC'')$ is a couple
  $\big(\sigma,\{\leq_c\}\big)$ consisting of a map $\sigma: C'\to
  C''$ of the underlying sets and of a family $\{\leq_c\}$ of total
  linear orders on each fiber $\sigma^{-1}(c)$, $c\in C''$, such that
  $\Lex\big(\sigma,\{\leq_c\}\big)=\CC'$.

  If $(C',\CC')$ is a partially cyclically ordered set, then a
  morphism is a compatible system of morphisms $(T,\CC'|_T)\to
  (C'',\CC'')$ given for all totally cyclically ordered subsets
  $T\subset C'$.
\end{definition}

\begin{example}
  A morphism of cyclically ordered sets is therefore a map of the
  underlying sets plus some additional data. We will however actually
  need only morphisms whose underlying map is an isomorphism. Since
  all its fibers are one-point sets, we do not need any additional
  data, and such a morphism will simply be an isomorphism $\sigma : C'
  \to C''$ of the underlying sets for which $(\sigma^{\times
    3})^{-1}(\CC'') = \CC'$.
\end{example}

From now on, by a cyclic order (without the adjective partial) we will
always mean a {\em total\/} cyclic order.
While linearly ordered sets $A$ are naturally represented by horizontal
intervals oriented from the left to the right, or as the combs
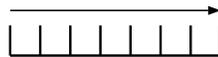
\begin{center}
\psscalebox{.4.4} 
{
\begin{pspicture}(0,-0.78)(7.08,0.78)
\psline[linecolor=black, linewidth=0.08](0.04,-0.74)(7.04,-0.74)
\psline[linecolor=black, linewidth=0.08](0.04,0.26)(0.04,-0.74)
\psline[linecolor=black, linewidth=0.08](1.04,0.26)(1.04,-0.74)(1.04,-0.74)
\psline[linecolor=black, linewidth=0.08](2.04,0.26)(2.04,-0.74)(2.04,-0.74)
\psline[linecolor=black, linewidth=0.08](3.04,0.26)(3.04,-0.74)(3.04,-0.74)
\psline[linecolor=black, linewidth=0.08](4.04,0.26)(4.04,-0.74)(4.04,-0.74)
\psline[linecolor=black, linewidth=0.08](5.04,0.26)(5.04,-0.74)(5.04,-0.74)
\psline[linecolor=black, linewidth=0.08](6.04,0.26)(6.04,-0.74)(6.04,-0.74)
\psline[linecolor=black, linewidth=0.08](7.04,0.26)(7.04,-0.74)(7.04,-0.74)
\psline[linecolor=black, linewidth=0.05, arrowsize=0.2cm 2.0,arrowlength=1.4,arrowinset=0.0]{<-}(7.04,0.76)(0.04,0.76)(0.04,0.76)
\end{pspicture}
}
\end{center}
whose teeth represent elements of $A$, 
we depict cyclically
ordered sets $C$ as circumferences of counterclockwise oriented circles
in the plane ${\mathbb R}^2$. We call such pictures {\em pancakes\/}:
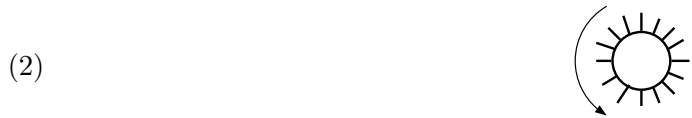
\begin{equation}
\label{the_pancake}
\raisebox{-1.5em}{}
\psscalebox{.4.4}
{
\begin{pspicture}(0,-.5)(4.0460277,1.7)
\pscircle[linecolor=black, linewidth=0.08, dimen=outer](2.4460278,-0.029330902){1.0}
\psline[linecolor=black, linewidth=0.08](3.1460278,0.6706691)(3.546028,1.070669)(3.546028,1.070669)
\psline[linecolor=black, linewidth=0.08](1.7460278,0.6706691)(1.3460279,1.070669)(1.3460279,1.070669)
\psline[linecolor=black, linewidth=0.08](0.9460278,-0.029330902)(1.4460279,-0.029330902)(1.4460279,-0.029330902)
\psline[linecolor=black, linewidth=0.08](2.4460278,-1.0293308)(2.4460278,-1.5293308)(2.4460278,-1.5293308)
\psline[linecolor=black, linewidth=0.08](3.1460278,-0.7293309)(3.546028,-1.1293309)(3.546028,-1.1293309)
\psline[linecolor=black, linewidth=0.08](2.4460278,1.570669)(2.4460278,0.9706691)(2.4460278,0.9706691)
\psarc[linecolor=black, linewidth=0.04, dimen=outer, arrowsize=0.2cm 2.0,arrowlength=1.4,arrowinset=0.0]{->}(2.3460279,-0.029330902){2.1}{120}{240}
\psline[linecolor=black, linewidth=0.08](3.3460279,-0.42933092)(3.8460279,-0.6293309)(3.8460279,-0.6293309)
\psline[linecolor=black, linewidth=0.08](2.8460279,-0.9293309)(3.046028,-1.429331)(3.046028,-1.429331)
\psline[linecolor=black, linewidth=0.08](2.046028,0.8706691)(1.8460279,1.4706692)(1.8460279,1.4706692)
\psline[linecolor=black, linewidth=0.08](2.8460279,0.8706691)(3.046028,1.3706691)(3.046028,1.3706691)
\psline[linecolor=black, linewidth=0.08](3.4460278,-0.029330902)(4.0460277,-0.029330902)(4.0460277,-0.029330902)
\psline[linecolor=black, linewidth=0.08](0.9460278,0.5706691)(1.5460278,0.3706691)(1.5460278,0.3706691)
\psline[linecolor=black, linewidth=0.08](1.6460278,-0.5293309)(1.1460278,-0.8293309)
\psline[linecolor=black, linewidth=0.08](2.046028,-0.8293309)(1.6460278,-1.429331)
\psline[linecolor=black, linewidth=0.08](3.8460279,0.6706691)(3.3460279,0.3706691)
\end{pspicture}
}
\end{equation}
with the spikes representing elements of $C$. Our
pancakes appeared as `spiders'
in~\cite{conant-vogtmann,markl-shnider-stasheff:book}, and, in the
form of $2k$-gons, as complementary regions of quasi-filling arc
systems in~\cite[Section~4]{Kaufmann-FeynmanII}.

We will need to extend the notation of
Definition~\ref{Kdy_prijede_Kveta_z_Rataj?} as follows.  For finite
ordered sets $\Rada A1k$ we denote by $A_1\cdots A_k$, their union
with the unique order in which $A_1 < \cdots < A_k$, and by
$[A_1\cdots A_k]$ the corresponding cyclically ordered set.  If e.g.\
$A_1 = \{u\}$ we abbreviate $[\{u\}A_2]$ by $[uA_2]$, \&c.

\begin{remark}
\label{zase_mam_nutkani}
While $A_1A_2
\not= A_2A_1$ when both $A_1$ and $A_2$ are non-empty,  $[A_1A_2]$ always equals
$[A_2A_1]$. Notice that $[A'] = [A'']$ for some finite ordered sets
$A',A''$ if and only if there are ordered sets  $A_1$ and $A_2$ such
that $A' = A_1A_2$ and $A'' = A_2A_1$.
\end{remark}

Let $\Fin$ (resp.~$\Cyc$) denote the category of finite
(resp.~cyclically ordered finite) sets and their
\underline{iso}morphisms.  As we recall in the Appendix, the pieces
$\oP\(S\)$ of an (ordinary) cyclic operad $\calP$ are indexed by
finite sets $S \in \Fin$, and their structure operations are
\begin{equation}
\label{eq:14}
\ooo{u}{v}:\oP\(S'\)
\otimes \oP\(S''\) \ \lra\ \oP\bl(S'\cup S'') \setminus \{u,v\}\br,
\end{equation}
where $S'$ and $S''$ are disjoint finite sets and $u\in S'$, $v \in S''$.

We will follow
the convention used in \cite{markl-shnider-stasheff:book} and distinguish \ns\
versions of operads by \underline{underl}y\underline{in}g.
A \ns\ cyclic operad~\cite[II.5.1]{markl-shnider-stasheff:book} $\nsP$ has its components 
$\nsP\(C\)$ indexed by cyclically ordered sets $C
\in \Cyc$, and structure operations 
\begin{equation}
\label{za_chvili_sednu_na_kolo_snad_uz}
\ooo{u}{v}:\nsP\(C'\)
\otimes \nsP\(C''\) \ \lra \ \nsP\bl(C'\cup C'') \setminus \{u,v\}\br
\end{equation}
of the same type as~(\ref{eq:14}). 
The codomain of~(\ref{za_chvili_sednu_na_kolo_snad_uz}) however does not make
sense unless we specify a cyclic order on the set
\begin{equation}
\label{eq:1}
(C' \cup C'') \setminus \{u,v\},
\end{equation}
were $C' \cup C''$ is the union of the corresponding underlying sets;
we will use this kind of shorthand freely.  It is given by the {\em
  pancake merging at $\{u,v\}$\/} as follows.

Assume that the cyclic order of $C'$ is
represented by the linear order
\[
a_1 < a_2 < \cdots < a_k < u
\]
and the cyclic order of $C''$ by
\[
v < b_1 < b_2 < \cdots < b_l.
\]
Then we equip~(\ref{eq:1}) with the cyclic order is represented by
\[
a_1 < a_2 < \cdots < a_k < b_1 < b_2 < \cdots < b_l.
\]

Notice that we allow the case when $C' = \stt u$ and $C'' = \stt v$,
then~(\ref{eq:1}) is an empty cyclically ordered set. In the pancake
world,~(\ref{eq:1}) is realized by merging two pancakes into one: 
\begin{equation}
\raisebox{-2em}{}
\label{pancake_surgery}
\psscalebox{.5.5}
{
\begin{pspicture}(0,-.4)(18.316029,1.84)
\pscircle[linecolor=black, linewidth=0.08, dimen=outer](2.4460278,0.020668488){1.0}
\pscircle[linecolor=black, linewidth=0.08, dimen=outer](6.4460278,0.020668488){1.0}
\psline[linecolor=black, linewidth=0.04, linestyle=dashed, dash=0.17638889cm 0.10583334cm](3.4460278,0.020668488)(5.4460278,0.020668488)(5.4460278,0.020668488)(5.4460278,0.020668488)
\psline[linecolor=black, linewidth=0.08](3.1460278,0.7206685)(3.546028,1.1206685)(3.546028,1.1206685)
\psline[linecolor=black, linewidth=0.08](1.7460278,0.7206685)(1.3460279,1.1206685)(1.3460279,1.1206685)
\psline[linecolor=black, linewidth=0.08](0.9460278,0.020668488)(1.4460279,0.020668488)(1.4460279,0.020668488)
\psline[linecolor=black, linewidth=0.08](2.4460278,-0.9793315)(2.4460278,-1.4793315)(2.4460278,-1.4793315)
\psline[linecolor=black, linewidth=0.06](3.1460278,-0.67933154)(3.546028,-1.0793315)(3.546028,-1.0793315)
\psline[linecolor=black, linewidth=0.06](1.7460278,-0.67933154)(1.3460279,-1.0793315)(1.3460279,-1.0793315)
\psline[linecolor=black, linewidth=0.08](6.4460278,1.5206685)(6.4460278,1.5206685)(6.4460278,1.0206685)
\psline[linecolor=black, linewidth=0.06](6.4460278,-0.9793315)(6.4460278,-1.4793315)(6.4460278,-1.4793315)
\psline[linecolor=black, linewidth=0.08](2.4460278,1.5206685)(2.4460278,1.0206685)(2.4460278,1.0206685)
\psline[linecolor=black, linewidth=0.06](5.246028,-0.9793315)(5.246028,-0.9793315)(5.646028,-0.5793315)
\psline[linecolor=black, linewidth=0.08](7.146028,0.7206685)(7.5460277,1.1206685)(7.5460277,1.1206685)
\psline[linecolor=black, linewidth=0.08](5.746028,0.7206685)(5.346028,1.1206685)(5.346028,1.1206685)
\psline[linecolor=black, linewidth=0.06](7.646028,-0.9793315)(7.246028,-0.5793315)(7.246028,-0.5793315)
\psline[linecolor=black, linewidth=0.08](13.046028,0.7206685)(13.446028,1.1206685)(13.446028,1.1206685)
\psline[linecolor=black, linewidth=0.08](11.646028,0.7206685)(11.246028,1.1206685)(11.246028,1.1206685)
\psline[linecolor=black, linewidth=0.06](10.846027,0.020668488)(11.346027,0.020668488)(11.346027,0.020668488)
\psline[linecolor=black, linewidth=0.06](12.346027,-0.9793315)(12.346027,-1.4793315)(12.346027,-1.4793315)
\psline[linecolor=black, linewidth=0.06](13.046028,-0.67933154)(13.446028,-1.0793315)(13.446028,-1.0793315)
\psline[linecolor=black, linewidth=0.06](11.646028,-0.67933154)(11.246028,-1.0793315)(11.246028,-1.0793315)
\psline[linecolor=black, linewidth=0.08](15.546028,1.5206685)(15.546028,1.5206685)(15.546028,1.0206685)
\psline[linecolor=black, linewidth=0.06](15.546028,-0.9793315)(15.546028,-1.4793315)(15.546028,-1.4793315)
\psline[linecolor=black, linewidth=0.08](12.346027,1.5206685)(12.346027,0.9206685)(12.346027,0.9206685)
\psline[linecolor=black, linewidth=0.08](14.846027,0.6206685)(14.446028,1.1206685)(14.846027,0.6206685)
\psline[linecolor=black, linewidth=0.06](16.746027,-0.9793315)(16.446028,-0.67933154)(16.446028,-0.67933154)
\psarc[linecolor=black, linewidth=0.08, dimen=outer](13.996028,1.0706685){0.95}{207.51837}{327.6011}
\psarc[linecolor=black, linewidth=0.08, dimen=outer](15.646028,0.020668488){1.0}{208}{147}
\psarc[linecolor=black, linewidth=0.08, dimen=outer](12.346027,0.020668488){1.0}{34.346157}{328.253}
\psarc[linecolor=black, linewidth=0.08, dimen=outer](13.996028,-1.0293316){0.95}{28.095282}{150.83333}
\psarc[linecolor=black, linewidth=0.04, dimen=outer, arrowsize=.2cm 2.0,arrowlength=1.4,arrowinset=0.0]{->}(2.3460279,0.020668488){2.1}{165}{311}
\rput[bl](0.9460278,1.5){\psscalebox{1.5 1.5}{$C'$}}
\rput[bl](9.7,1.5){\psscalebox{1.5 1.5}{$C' \setminus \stt u $}}
\rput[bl](5,1.5){\psscalebox{1.5 1.5}{$C''$}}
\rput[br](17.946028,1.5){\psscalebox{1.5 1.5}{$C''\setminus \stt v$}}
\rput[bl](8.4,-.23){\psscalebox{2.5 2.5}{$\Longrightarrow$}}
\rput[br](4,0.2206685){\psscalebox{1.5 1.5}{$u$}}
\rput[bl](4.9,0.2206685){\psscalebox{1.5 1.5}{$v$}}
\psline[linecolor=black, linewidth=0.08](16.346027,0.7206685)(16.346027,0.7206685)(16.746027,1.1206685)
\psline[linecolor=black, linewidth=0.06](14.446028,-0.9793315)(14.446028,-0.9793315)(14.846027,-0.5793315)
\psline[linecolor=black, linewidth=0.06](16.646029,0.020668488)(17.246027,0.020668488)(17.246027,0.020668488)
\psline[linecolor=black, linewidth=0.06](7.4460278,0.020668488)(7.9460278,0.020668488)(7.9460278,0.020668488)
\psdots[linecolor=black, dotstyle=o, dotsize=0.23125](5.4460278,0.020668488)
\rput{92.38133}(3.6098614,-3.4215245){\psdots[linecolor=black, dotstyle=o, dotsize=0.23125](3.4460278,0.020668488)}
\end{pspicture}
}
\end{equation}

Modular operads~\cite[Section~2]{getzler-kapranov:CompM98},
\cite[Section~II.5.3]{markl-shnider-stasheff:book} have,
besides~(\ref{eq:14}), also the contractions
\begin{equation*}
\label{eq:15}
\xi_{uv}  : \oP\(S;g\) \to \oP\bl S \setminus
\{u,v\};g+1\br,
\end{equation*}
where $S\in \Fin$ and $u,v \in S$ are distinct elements; $\oP$ here
has an additional grading by the genus $g \in \bbN$ which is irrelevant
now.  It is therefore natural to expect that our conjectural \ns\
modular operad has pieces $\nsP(C;g)$ indexed by $C \in \Cyc$, $g \in
\bbN$ and, besides~(\ref{za_chvili_sednu_na_kolo_snad_uz}), the
contractions
\begin{equation}
\label{eq:15bis}
\xi_{uv} : \nsP\(C;g\) \to \nsP\bl C \setminus \{u,v\};g+1\br.
\end{equation}
There is only one natural cyclic order on the subset $C \setminus
\{u,v\}$ of $C$, namely the restriction of the cyclic
order of $C$, so we are forced to equip~(\ref{eq:1}) with
this cyclic order. 
The following example shows that it does not work.

\begin{example}
\label{dead_end}
Consider ordered sets $X$, $Y$ and $Z$, and distinct symbols $u',u'',v'$ and
$v''$. Let
\[
x \in \nsP\bl[Xv'Zu'];g'\br \ \hbox{ and } \   y \in
\nsP\bl[Yu''v''];g''\br
\]
be arbitrary elements.
According to the definition of the cyclic order of~(\ref{eq:1}) used 
in~(\ref{za_chvili_sednu_na_kolo_snad_uz}),
\[
(x \ooo {v'}{v''} y) \in \nsP\bl[Zu'XYu''];g'\sqplus g''\br\  \hbox { and } \
(x \ooo {u'}{u''} y) \in \nsP\bl[Xv'Zv''Y];g' \sqplus g'' \br,
\]
thus
\[
\xi_{u'u''}(x \ooo {v'}{v''} y) \in \nsP\bl[{\it XYZ}];g' \sqplus g''
\sqplus1
\br \
\hbox { while } \
\xi_{v'v''}(x \ooo {u'}{u''} y) \in \nsP\bl[{\it YXZ}]; 
g' \sqplus g'' \sqplus1\br.
\] 
The standard exchange rule in Definition~\ref{modular}(iv) between compositions and contractions in
a modular operad must of course hold also in the \ns\ case, therefore
\begin{equation}
\label{za_chvili_prijde_rodina}
\xi_{u'u''}(x \ooo {v'}{v''} y) = \xi_{v'v''}(x \ooo {u'}{u''} y). 
\end{equation}
But this is  {\bf not  possible}. If  $X,Y,Z \not= \emptyset$,
the cyclically ordered sets $[{\it XYZ}]$ and $[{\it XZY}]$ are non-isomorphic,
so~(\ref{za_chvili_prijde_rodina}) compares elements of different
spaces.  This quandary will be resolved
by introducing multicyclically ordered sets.
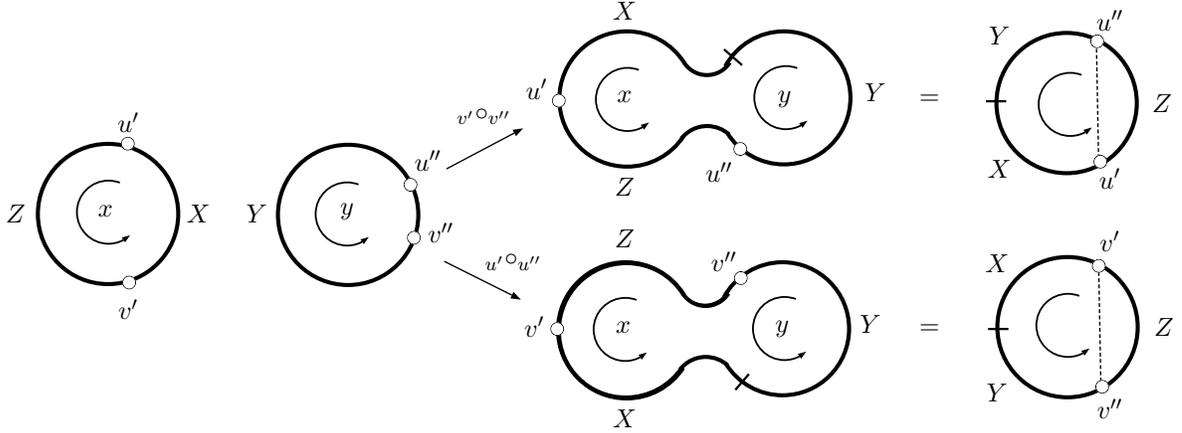
\begin{figure}
\psscalebox{.3.3} 
{
\begin{pspicture}(0,-16.5)(50,1)
\psarc[linecolor=black, linewidth=0.18, dimen=outer](26.22,-2.835658){3.0}{31.845161}{326.7251}
\psarc[linecolor=black, linewidth=0.18,
dimen=outer](33.17,-2.785658){2.95}{212.0579}{151.71265}
\psarc[linecolor=black, linewidth=0.18, dimen=outer](29.77,-5.485658){1.45}{48.062286}{142.09}
\psarc[linecolor=black, linewidth=.08, dimen=outer, arrowsize=0.1cm
2.0,arrowlength=1.4,arrowinset=0.0]{->}(33.26,-2.775658){1.4}{69.026505}{313.13522}
\psarc[linecolor=black, linewidth=.08, dimen=outer, arrowsize=0.1cm 2.0,arrowlength=1.4,arrowinset=0.0]{->}(26.24,-2.855658){1.4}{69.026505}{313.13522}
\psarc[linecolor=black, linewidth=0.18, dimen=outer](29.76,-0.57565796){1.2}{213.06888}{327.65515}
\psline[linecolor=black, linewidth=.1](30.52,-0.63565797)(31.32,-1.335658)
\rput(26.12,-2.735658){\psscalebox{3.0 3.0}{$x$}}
\rput(33.2,-2.8156579){\psscalebox{3.0 3.0}{$y$}}
\rput[b](22.22,-2.835658){\psscalebox{3.0 3.0}{$u'$}}
\rput[tr](30.9,-5.635658){\psscalebox{3.0 3.0}{$u''$}}
\rput[t](26.12,-6.335658){\psscalebox{3.0 3.0}{$Z$}}
\rput(26.12,1.064342){\psscalebox{3.0 3.0}{$X$}}
\rput(37.22,-2.535658){\psscalebox{3.0 3.0}{$Y$}}


\psarc[linecolor=black, linewidth=.18,
dimen=outer](26.16,-13.075658){3.0}{31.845161}{326.7251}
\psarc[linecolor=black, linewidth=.18, dimen=outer](33.11,-13.025658){2.95}{212.0579}{151.71265}
\psarc[linecolor=black, linewidth=0.12,
dimen=outer](29.66,-10.775658){1.2}{213.06888}{327.65515}
\psarc[linecolor=black, linewidth=0.08, dimen=outer, arrowsize=0.1cm
2.0,arrowlength=1.4,arrowinset=0.0]{->}(33.2,-13.015658){1.4}{69.026505}{313.13522}
\psarc[linecolor=black, linewidth=0.08, dimen=outer, arrowsize=0.1cm 2.0,arrowlength=1.4,arrowinset=0.0]{->}(26.14,-13.055658){1.4}{69.026505}{313.13522}
\psarc[linecolor=black, linewidth=.18,
dimen=outer](26.2,-13.115658){3.0}{31.845161}{326.7251}
\psarc[linecolor=black, linewidth=.18,
dimen=outer](29.7,-10.815658){1.2}{213.06888}{327.65515}
\psarc[linecolor=black, linewidth=.18, dimen=outer](29.71,-15.725658){1.45}{48.062286}{142.09}
\psline[linecolor=black, linewidth=.1](31.62,-15.035658)(31.02,-15.735658)
\rput(26.06,-12.975658){\psscalebox{3.0 3.0}{$x$}}
\rput(33.1,-13.015658){\psscalebox{3.0 3.0}{$y$}}
\rput(22.12,-12.9356575){\psscalebox{3.0 3.0}{$v'$}}
\rput(26.12,-9.035658){\psscalebox{3.0 3.0}{$Z$}}
\rput(30.52,-10.035658){\psscalebox{3.0 3.0}{$v''$}}
\rput(37,-12.8){\psscalebox{3.0 3.0}{$Y$}}
\rput[t](26.12,-16.6){\psscalebox{3.0 3.0}{$X$}}

\rput(-1,0){
\pscircle[linecolor=black, linewidth=0.18, dimen=outer](4.22,-7.935658){3.2}
\psarc[linecolor=black, linewidth=0.08, dimen=outer, arrowsize=0.1cm
2.0,arrowlength=1.4,arrowinset=0.0]{->}(4.24,-7.855658){1.4}{69.026505}{313.13522}
\rput(4.1,-7.815658){\psscalebox{3.0 3.0}{$x$}}
\rput(0.12,-7.935658){\psscalebox{3.0 3.0}{$Z$}}
\rput(8.22,-7.935658){\psscalebox{3.0 3.0}{$X$}}
\rput[b](5.12,-4.3){\psscalebox{3.0 3.0}{$u'$}}
\rput[t](5.12,-11.7){\psscalebox{3.0 3.0}{$v'$}}
}

\pscircle[linecolor=black, linewidth=0.18,
dimen=outer](13.86,-7.975658){3.2}
\psarc[linecolor=black, linewidth=.08, dimen=outer, arrowsize=0.1cm
2.0,arrowlength=1.4,arrowinset=0.0]{->}(13.82,-7.935658){1.4}{69.026505}{313.13522}
\rput(13.82,-7.835658){\psscalebox{3.0 3.0}{$y$}}
\rput[b](17.4,-6){\psscalebox{3.0 3.0}{$u''$}}
\rput[lb](17.4,-9.2){\psscalebox{3.0 3.0}{$v''$}}
\rput(9.82,-7.935658){\psscalebox{3.0 3.0}{$Y$}}

\pscircle[linecolor=black, linewidth=0.18, dimen=outer](45.7,-3.015658){3.2}
\psarc[linecolor=black, linewidth=.08, dimen=outer, arrowsize=0.1cm
2.0,arrowlength=1.4,arrowinset=0.0]{->}(45.86,-3.0756578){1.4}{69.026505}{313.13522}
\psline[linecolor=black, linewidth=0.06, linestyle=dashed, dash=0.17638889cm 0.10583334cm](47.02,-0.33565795)(47.12,-5.735658)
\psline[linecolor=black, linewidth=.1](42.12,-2.935658)(43.02,-2.935658)

\pscircle[linecolor=black, linewidth=0.18, dimen=outer](45.74,-12.955658){3.2}
\psarc[linecolor=black, linewidth=.08, dimen=outer, arrowsize=0.1cm 2.0,arrowlength=1.4,arrowinset=0.0]{->}(45.76,-12.875658){1.4}{69.026505}{313.13522}
\psline[linecolor=black, linewidth=0.06, linestyle=dashed, dash=0.17638889cm 0.10583334cm](47.12,-10.235658)(47.22,-15.635658)
\psline[linecolor=black, linewidth=.1    ](42.22,-13.035658)(43.12,-13.035658)

\psdots[linecolor=black, fillstyle=solid, dotstyle=o, dotsize=0.6](23.12,-13.035658)
\psdots[linecolor=black, fillstyle=solid, dotstyle=o, dotsize=0.6](31.26,-10.775658)
\psdots[linecolor=black, fillstyle=solid, dotstyle=o, dotsize=0.6](23.2,-2.915658)
\psdots[linecolor=black, fillstyle=solid, dotstyle=o, dotsize=0.6](31.24,-5.055658)
\psdots[linecolor=black, fillstyle=solid, dotstyle=o, dotsize=0.6](16.78,-8.995658)
\psdots[linecolor=black, fillstyle=solid, dotstyle=o, dotsize=0.6](16.62,-6.635658)
\psdots[linecolor=black, fillstyle=solid, dotstyle=o, dotsize=0.6](4.16,-10.975658)
\psdots[linecolor=black, fillstyle=solid, dotstyle=o, dotsize=0.6](4.1,-4.815658)
\psdots[linecolor=black, fillstyle=solid, dotstyle=o, dotsize=0.6](47.28,-15.595658)
\psdots[linecolor=black, fillstyle=solid, dotstyle=o, dotsize=0.6](47.12,-10.235658)
\psdots[linecolor=black, fillstyle=solid, dotstyle=o, dotsize=0.6](47.06,-0.27565795)
\psdots[linecolor=black, fillstyle=solid, dotstyle=o, dotsize=0.6](47.2,-5.615658)

\rput(42.72,-0.035657957){\psscalebox{3.0 3.0}{$Y$}}
\rput(42.72,-5.935658){\psscalebox{3.0 3.0}{$X$}}
\rput(42.62,-10.135658){\psscalebox{3.0 3.0}{$X$}}
\rput(42.62,-15.835658){\psscalebox{3.0 3.0}{$Y$}}
\rput(49.92,-3.035658){\psscalebox{3.0 3.0}{$Z$}}
\rput(47.62,0.64342){\psscalebox{3.0 3.0}{$u''$}}
\rput(47.62,-6.35658){\psscalebox{3.0 3.0}{$u'$}}
\rput(47.62,-9.135658){\psscalebox{3.0 3.0}{$v'$}}
\rput(47.62,-16.5658){\psscalebox{3.0 3.0}{$v''$}}
\rput(50.02,-12.9356575){\psscalebox{3.0 3.0}{$Z$}}
\rput(19.8,-3.6){\psscalebox{3.0 3.0}{$\ooo {v'}{v''}$}}
\rput[b](21.1,-10.4){\psscalebox{3.0 3.0}{$\ooo {u'}{u''}$}}
\rput(39.62,-2.835658){\psscalebox{3.0 3.0}{$=$}}
\rput(39.62,-13.035658){\psscalebox{3.0 3.0}{$=$}}

\psline[linecolor=black, linewidth=.06, arrowsize=0.2cm 2.0,arrowlength=1.4,arrowinset=0.0]{<-}(21.62,-4.135658)(18.22,-5.935658)
\psline[linecolor=black, linewidth=0.06, arrowsize=0.2cm 2.0,arrowlength=1.4,arrowinset=0.0]{<-}(21.52,-11.735658)(18.12,-10.035658)
\end{pspicture}
}
\caption{\label{fig:quandary}
Na\"\i ve attempts fail: the relative positions of $X$
  and $Y$ decorating the circumferences of the two pancakes in the last
  column are interchanged.}
\end{figure}

We believe that Figure~\ref{fig:quandary} helps to understand the
situation.  It shows (from the left to the right the pancake
representing the cyclically ordered set $C' = [Xv'Zu']$, the one
representing $C'' = [Yu''v'']$ and two realizations of the pancakes
representing the merging of $C' \cup C''$ at $\{v',v''\}$ resp.\ the
merging $C' \cup C''$ at $\{u',u''\}$. The meaning of the dotted lines
will be explained in Example~\ref{zas_mne_poboliva_v_krku}.
\end{example}

\section{Multicyclic orders}\label{multord}

In this section we introduce multicyclically ordered sets as objects
appropriately indexing the pieces of \ns\ modular operads.

\begin{definition}
\label{sec:multicyclic-orders-1}
A {\em multicyclic order\/} on a set finite $S$ is a disjoint decomposition $\sfS = C_1 \cup
\cdots \cup C_b$ of $S$ into $b > 0$ possibly empty totally cyclically
ordered sets. 
A {\em morphism\/} 
\[
\sigma: \sfS' = C'_1 \cup \cdots \cup C'_{b'}
\longrightarrow \sfS'' = C''_1 \cup \cdots \cup C'_{b''}
\]
is a couple
$(\sigma,u)$ consisting of

\begin{itemize}
\item[(i)]
a morphism
$\sigma = \big(\sigma,\{\leq_s\}\big)    : S' \to S''$ of
the underlying sets with the induced cyclic orders, and of

\item[(ii)]
a map $u : \{1,\ldots,b'\} \to  \{1,\ldots,b''\}$ of the indexing sets 
\end{itemize}
such that
 that $\sigma(C'_i) \subset C''_{u(i)}$
for each $1\leq i \leq b'$. 
\end{definition} 

Notice that a given set has infinitely many
multicyclic orders, but the geometricity  that we
introduce in Definition~\ref{sec:multicyclic-orders} below allows only
finite number of them. 

\begin{remark}
\label{sec:multicyclic-orders-3}
  It is clear that a morphism in
  Definition~\ref{sec:multicyclic-orders-1} is given by a family
\[
\sigma_i = \big(\sigma_i,\{\leq_{c_i}\}\big) : C'_i \to C''_{u(i)},\
 1\leq i \leq b',
\]
of morphisms of totally cyclically ordered sets. This offers the
following alternative description.  For a category $\catC$ denote by
$\Mult(\catC)$ the category of formal finite coproducts $A_1 \sqcup
\cdots \sqcup A_s$, $s \geq 1$, of objects of $\catC$ with the
Hom-sets
\[
\Mult(\catC)\big(A_1 \sqcup \cdots \sqcup A_s,B_1 \sqcup \cdots \sqcup
B_t\big) :=
\prod_{1 \leq i \leq s}\coprod_{1 \leq j \leq t}
\catC(A_i,B_j).
\]
Morphisms of Definition~\ref{sec:multicyclic-orders-1} are precisely
morphism in the category $\Mult({\mathbf \Lambda})$ generated by
the category ${\mathbf \Lambda}$ of finite totally cyclically
ordered sets.
\end{remark}

We denote by $\MultCyc$ the category of multicyclically ordered sets
and their {\bf iso}morphisms. It contains the full subcategory $\Cyc$
of cyclically ordered finite sets and their isomorphisms 
embedded as multicyclically ordered
sets with $b =1$. 

\begin{remark}
\label{sec:multicyclic-orders-2}
It follows from the observations in
Remark~\ref{sec:multicyclic-orders-3} that 
\[
\MultCyc \cong \Mult(\Cyc),
\]
our basic category of multicyclic sets could therefore be introduced
using only isomorphisms of totally cyclically ordered sets.
\end{remark}

\begin{remark}
  The underlying set $S$ of each $\sfS = C_1\cup \cdots \cup C_b \in
  \MultCyc$ has a partial cyclic order $\CC_1 \cup \cdots \cup \CC_b$
  induced by the cyclic orders $\CC_i$ of $C_i$, $1\leq i \leq
  b$. Notice however that the correspondence $\sfS \mapsto (S,\CC_1
  \cup \cdots \cup \CC_b)$ does not induce an embedding of $\MultCyc$
  into the category of partially cyclically ordered sets. If, for
  instance, $\sfS' = C_1$ is a non-empty totally cyclically ordered
  set and $\sfS'' = C_1\cup C_2$ with $C_2 = \emptyset$, then the
  underlying cyclically ordered sets agree, but $\sfS'$ and $\sfS''$
  are not isomorphic in $\MultCyc$.
\end{remark}

\begin{definition}
  A {\em \ns\ modular module\/} is a functor
 \[
E : \MultCyc \times \bbN  \rightarrow \ttM,
\]
where \ttM\ is our fixed symmetric monoidal category 
and the natural numbers $\bbN = \{0,1,2,\ldots\}$ are considered as a
discrete category.
\end{definition}

Explicitly, a \ns\ modular module is a rule $(\sfS,g)\mapsto
E\(\sfS;g\)$ that assigns  to each multicyclically ordered $\sfS$ and
 $g \in \bbN$ an object $E\(\sfS,g\) \in \ttM$,
together with a functorial family of maps $E\(\sigma\) : E\(\sfS',g\) \to
E\(\sfS'',g\)$ defined for each isomorphism $\sigma : \sfS' \to \sfS''$
of multicyclically ordered sets. If $\sfS = C_1 \cup \cdots\cup C_b$, we
will sometimes write more explicitly $E\(\Rada C1b;g\)$ instead of
$E\(\sfS;g\)$. We call $g$ the (operadic) {\em genus\/}.

\begin{definition}
\label{sec:multicyclic-orders}
We call a couple $(\sfS,g) \in \MultCyc \times \bbN$ with $\sfS =
C_1 \cup \cdots\cup C_b$ {\em geometric\/} if
\begin{equation}
\textstyle  
\label{eq:Jaruska}
G :=\frac 12( g- b +1) \in \bbN.   
\end{equation}
A \ns\ modular module $E$ is {\em geometric\/}, if  $E\(\sfS;g\)
\not=0$ implies that $(\sfS,g)$ is geometric.
\end{definition}

Geometricity therefore means that $g- b +1$ is an even non-negative
integer. We will call $G$ defined
in~(\ref{eq:Jaruska}) the {\em geometric genus\/} and $b$ the {\em number
of boundaries\/}. 
The reasons for this terminology
will be clarified later in Section~\ref{lazarev}, see also Remark~\ref{why_geometric}.

\begin{example}
For $g \leq 3$, only the following components of a geometric \ns\
modular module can be nontrivial:
 $E\(C_1;0\),\ E\(C_1,C_2;1\),\  E\(C_1,C_2,C_3;2\)$ and $E\(C_1,C_2,C_3,C_4;3\)$
 in the geometric genus $0$, and $E\(C_1;2\),\ 
E\(C_1,C_2;3\)$ in geometric genus $1$.
\end{example}

Let $C$ be a cyclically ordered set. For $u,v \in C$, the set $C\setminus \{u,v\}$
has an obvious induced cyclic order given by the
restriction of the original one. It naturally decomposes as
\begin{equation}
  \label{eq:3}
 C \setminus \{u,v\} = I_1 \cup I_2
\end{equation}
where $I_1$, $I_2$ are the open intervals whose boundary points are
$u$ and $v$, considered with the induced cyclic orders. Notice that if
$u$ and $v$ are adjacent in the cyclic order, one or
both of $I_1$, $I_2$ may be empty.

If we distribute the elements of $C$ 
around the circumference of a  pancake so that their
cyclic order is induced by the (say) counterclockwise orientation of
the oven plate, then~(\ref{eq:3}) is realized by cutting the pancake, 
with the knife running through $u$ and $v$, as in:
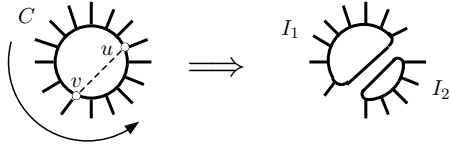
\begin{center}
\psscalebox{.5.5} 
{
\begin{pspicture}(0,-2.3344152)(11.346027,2.3344152)
\pscircle[linecolor=black, linewidth=0.08, dimen=outer](2.4460278,-0.029330902){1.0}
\psline[linecolor=black, linewidth=0.08](3.1460278,0.6706691)(3.546028,1.070669)(3.546028,1.070669)
\psline[linecolor=black, linewidth=0.08](1.7460278,0.6706691)(1.3460279,1.070669)(1.3460279,1.070669)
\psline[linecolor=black, linewidth=0.08](0.9460278,-0.029330902)(1.4460279,-0.029330902)(1.4460279,-0.029330902)
\psline[linecolor=black, linewidth=0.08](2.4460278,-1.0293308)(2.4460278,-1.5293308)(2.4460278,-1.5293308)
\psline[linecolor=black, linewidth=0.08](3.1460278,-0.7293309)(3.546028,-1.1293309)(3.546028,-1.1293309)
\psline[linecolor=black, linewidth=0.08](2.4460278,1.570669)(2.4460278,0.9706691)(2.4460278,0.9706691)
\psarc[linecolor=black, linewidth=0.04, dimen=outer, arrowsize=0.2cm 2.0,arrowlength=1.4,arrowinset=0.0]{->}(2.3460279,-0.029330902){2.1}{165}{311}
\rput[bl](0.5,1){\psscalebox{1.5 1.5}{$C$}}
\rput[bl](2.7,0.1706691){\psscalebox{1.5 1.5}{$u$}}
\rput[bl](1.9,-0.7){\psscalebox{1.5 1.5}{$v$}}
\rput[br](8,0.6){\psscalebox{1.5 1.5}{$I_1$}}
\rput[bl](11.5,-1){\psscalebox{1.5 1.5}{$I_2$}}
\rput[bl](5,-.4){\psscalebox{2.5 2.5}{$\Longrightarrow$}}
\psdots[linecolor=black, fillstyle=solid, dotstyle=o, dotsize=0.23125](2.046028,-0.9293309)
\psdots[linecolor=black, fillstyle=solid, dotstyle=o, dotsize=0.23125](3.3460279,0.3706691)
\psline[linecolor=black, linewidth=0.08](3.3460279,-0.42933092)(3.8460279,-0.6293309)(3.8460279,-0.6293309)
\psline[linecolor=black, linewidth=0.08](2.8460279,-0.9293309)(3.046028,-1.429331)(3.046028,-1.429331)
\psline[linecolor=black, linewidth=0.08](2.046028,0.8706691)(1.8460279,1.4706692)(1.8460279,1.4706692)
\psline[linecolor=black, linewidth=0.08](2.8460279,0.8706691)(3.046028,1.3706691)(3.046028,1.3706691)
\psline[linecolor=black, linewidth=0.08](10.446028,0.6706691)(10.846027,0.8706691)(10.846027,0.8706691)
\psline[linecolor=black, linewidth=0.08](9.046028,0.6706691)(8.646028,1.070669)(8.646028,1.070669)
\psline[linecolor=black, linewidth=0.08](8.246028,-0.029330902)(8.746028,-0.029330902)(8.746028,-0.029330902)
\psline[linecolor=black, linewidth=0.08](9.746028,-1.0293308)(9.746028,-1.5293308)(9.746028,-1.5293308)
\psline[linecolor=black, linewidth=0.06](10.446028,-0.7293309)(10.846027,-1.1293309)(10.846027,-1.1293309)
\psline[linecolor=black, linewidth=0.08](9.746028,1.570669)(9.746028,0.9706691)(9.746028,0.9706691)

\psline[linecolor=black, linewidth=0.08](10.646028,-0.42933092)(11.146028,-0.6293309)(11.146028,-0.6293309)
\psline[linecolor=black, linewidth=0.08](10.146028,-0.9293309)(10.346027,-1.429331)(10.346027,-1.429331)
\psline[linecolor=black, linewidth=0.08](8.846027,0.3706691)(8.346027,0.6706691)(8.346027,0.6706691)
\psline[linecolor=black, linewidth=0.08](9.346027,0.8706691)(9.146028,1.4706692)(9.146028,1.4706692)
\psline[linecolor=black, linewidth=0.08](10.146028,0.8706691)(10.346027,1.3706691)(10.346027,1.3706691)
\psline[linecolor=black, linewidth=0.08](3.4460278,0.3706691)(3.9460278,0.5706691)(3.9460278,0.5706691)
\psline[linecolor=black, linewidth=0.04, linestyle=dashed, dash=0.17638889cm 0.10583334cm](2.1460278,-0.8293309)(3.246028,0.2706691)(3.246028,0.2706691)
\psline[linecolor=black, linewidth=0.08](3.4460278,-0.029330902)(4.0460277,-0.029330902)(4.0460277,-0.029330902)
\psline[linecolor=black, linewidth=0.08](0.9460278,0.5706691)(1.5460278,0.3706691)(1.5460278,0.3706691)
\psline[linecolor=black, linewidth=0.08](10.746028,-0.029330902)(11.346027,-0.029330902)(11.346027,-0.029330902)
\psline[linecolor=black, linewidth=0.08](1.6460278,-0.5293309)(1.1460278,-0.8293309)
\psline[linecolor=black, linewidth=0.08](8.946028,-0.5293309)(8.646028,-0.9293309)
\psline[linecolor=black, linewidth=0.08](10.446028,-0.7293309)(10.846027,-1.1293309)
\psline[linecolor=black, linewidth=0.08](1.9460279,-1.0293308)(1.6460278,-1.429331)
\psarc[linecolor=black, linewidth=0.08, dimen=outer](9.696028,0.020669097){0.95}{38.997017}{210.56715}
\psarc[linecolor=black, linewidth=0.08, dimen=outer](9.796028,-0.0793309){0.95}{274.397}{355.44293}
\psarc[linecolor=black, linewidth=0.08, dimen=outer](10.246028,0.5706691){0.2}{299}{60}
\psarc[linecolor=black, linewidth=0.08, dimen=outer](9.096027,-0.3793309){0.25}{199.6999}{319.36636}
\psarc[linecolor=black, linewidth=0.08, dimen=outer](10.596027,-0.0793309){0.15}{341}{145}
\psarc[linecolor=black, linewidth=0.08, dimen=outer](9.796028,-0.8793309){0.15}{129.16226}{281.66156}
\psline[linecolor=black, linewidth=0.08](10.346027,0.3706691)(9.246028,-0.6293309)
\psline[linecolor=black, linewidth=0.08](10.546028,0.0706691)(9.646028,-0.8293309)
\psline[linecolor=black, linewidth=0.08](9.746028,-1.0293308)(9.946028,-1.0293308)(9.946028,-1.0293308)
\psline[linecolor=black, linewidth=0.08](10.746028,-0.029330902)(10.746028,-0.2293309)
\end{pspicture}
}
\end{center}

Pancake cutting together with pancake merging~(\ref{pancake_surgery})  
induce two basic
operations on multicyclically ordered sets.
The {\em merging\/} starts with two multicyclically ordered sets 
\[
\sfS' = C'_1\cup
\cdots \cup C'_{b'} \ \mbox{ and } \ \sfS'' = C''_1\cup \cdots \cup
C''_{b''},
\]
whose underlying sets $S'$ and $S''$ are disjoint. For $u \in C'_i$
and $v \in C''_j$, let 
\[
\sfS' \cup \sfS'' \setminus \stt {u,v}
\]
be the multicyclically ordered set whose 
underlying set is $S' \cup S'' \setminus \stt {u,v}$, decomposed as
\begin{equation}
\label{eq:4}
C'_1\cup  \cdots\cup \widehat
{C'_i} \cup \cdots \cup C'_{b'} \cup 
C''_1\cup \cdots \widehat {C''_j} \cup \cdots \cup C''_{b''} \cup
\big(C'_i \cup C''_j \setminus \stt {u,v}\big),
\end{equation}
where $\widehat{\hphantom{C}}$ indicates that the corresponding term
has been omitted, and $\big(C'_i \cup C''_j \setminus \stt {u,v}\big)$
is cyclically ordered as in~(\ref{eq:1}).

Let
$\sfS = C_1\cup \cdots \cup C_{b}$ is a multicyclically ordered set,  $u \in C_i$, $v \in C_j$. If $i
\not= j$, we define the {\em cut\/}
\begin{equation}
\label{Za_chvili_s_Jaruskou_v_levnem_baru}
\sfS  \setminus \stt {u,v}
\end{equation}
to be the multicyclically ordered set whose 
underlying set $S  \setminus \stt {u,v}$ decomposed as
\begin{equation}
\label{eq:5}
S\setminus \stt {u,v} = C_1\cup  \cdots\cup \widehat
{C_i} \cup \cdots \widehat {C_j} \cup \cdots \cup C_{b} \cup
\big(C_i \cup C_j \setminus \stt {u,v}\big)
\end{equation}
with $C_i \cup C_j \setminus \stt {u,v}$ cyclically ordered
as in~(\ref{eq:1}).  If $i =
j$, we define~(\ref{Za_chvili_s_Jaruskou_v_levnem_baru}) as the multicyclically
ordered set given by the decomposition
\begin{equation}
\label{eq:6}
S\setminus \stt {u,v} = C_1\cup  \cdots\cup \widehat
{C_i}\cup \cdots \cup C_{b} \cup
\big(C_i \setminus \stt {u,v}\big),
\end{equation}
where $C_i \setminus \stt {u,v}$ is the union of two multicyclically
ordered sets as in~(\ref{eq:3}). Notice that the number of cyclically
ordered components of~(\ref{eq:4}) is $b'+b'' -1$, of~(\ref{eq:5}) is
$b-1$ and of~(\ref{eq:6}) is $b+1$.

\section{Biased definition of non-$\Sigma$ modular operads.}
\label{sec:bias-defin-non}

We formulate a definition of \ns\ modular operads
biased towards the bilinear operations $\ooo uv$ and contractions
$\xi_{uv}$.  Recall that $\ttM$ denotes our basic symmetric monoidal
category; let $\tau$ be its commutativity constraint. Regarding
multicyclically ordered sets, we use the notation introduced in
\S\ref{multord}. 

\begin{definition}
\label{opet_podleham}
A {\ns\ modular operad\/} in $\ttM = (\ttM,\ot,1)$  is a \ns\ modular module
\[
\nsP = \big\{\nsP\(\sfS;g\) \in \ttM;\  (\sfS,g)   \in  \MultCyc \times \bbN \big\}
\]
together with morphisms (compositions)
\begin{equation}
\label{eq:7}
\ooo{u}{v}:\nsP\(\sfS';g'\)
\otimes \nsP\(\sfS'';g''\)  \to \nsP\bl\sfS'\cup \sfS'' \setminus \{u,v\};g'+g''\br
\end{equation}
defined for arbitrary disjoint multicyclically ordered 
sets $\sfS'$ and $\sfS''$ with elements $u \in S'$,
$v \in S''$ of their underlying sets, and contractions
\begin{equation}
\label{eq:8}
\xi_{uv} = \xi_{vu} : \nsP\(\sfS;g\) \to \nsP\bl\sfS \setminus \{u,v\};g+1\br
\end{equation}
given for any multicyclically ordered 
set $\sfS$ and distinct elements $u,v \in S$ of its underlying set.
These data are required to satisfy the following axioms.
\begin{enumerate}
\itemindent -1em
 \itemsep .3em 
\item [(i)]
For $\sfS'$, $\sfS''$ and $u$, $v$ as in~(\ref{eq:7}),
one has the equality
\[
\ooo{u}{v} =  \ooo{v}{u} \tau
\]
of maps $\nsP\(\sfS';g'\)\otimes \nsP\(\sfS'';g''\) \to
\nsP\bl\sfS'\cup \sfS'' 
\setminus  \{u,v\};g'+g''\br$.

\item  [(ii)]
For mutually disjoint multicyclically ordered sets
  $\sfS_1,\sfS_2,\sfS_3$, 
and $a \in S_1$, $b,c \in
  S_2$, $b \not= c$, $d  \in S_3$, one has the equality
\[
\ooo ab (\id \ot \ooo cd)  = \ooo cd (\ooo ab \ot
\id)
\] 
of maps $\nsP\(\sfS_1;g_1\) \otred \nsP\(\sfS_2;g_2\) \otred  \nsP\(\sfS_3;g_3\) \to 
\nsP\bl\sfS_1 \cupred \sfS_2 \cupred \sfS_3 \setminus \{a,b,c,d\};
g_1\!+\!
g_2\! +\! g_3\br$.

\item  [(iii)]
For a multicyclically ordered set $\sfS$ and mutually distinct $a,b,c,d \in S$,
one has the equality
\[
\xi_{ab} \ \xi_{cd} = \xi_{cd} \ \xi_{ab}
\]
of maps $\nsP\(\sfS;g\) \to \nsP\bl\sfS \setminus \{a,b,c,d\};g+2\br$.

\item  [(iv)]
For multicyclically ordered sets $\sfS', \sfS''$ and distinct $a,c \in
S'$, $b,d \in S''$, one has the equality
\[
\xi_{ab} \ \ooo{c}{d} = \xi_{cd} \ \ooo{a}{b}
\]
of maps $\nsP\(\sfS' \cup \sfS'';g\) \to \nsP\bl\sfS' \cup \sfS'' \setminus
\{a,b,c,d\};g+1\br$.

\item [(v)] For multicyclically ordered sets $\sfS', \sfS''$ and
  mutually distinct $a,c,d \in S'$, $b \in S''$, one has the equality
\[
\ooo{a}{b} \ (\xi_{cd}\ot\id) = \xi_{cd} \ \ooo{a}{b}
\]
of maps $\nsP\(\sfS' \cup \sfS'';g\) \to \nsP\bl\sfS' \cup \sfS'' \setminus
\{a,b,c,d\};g+1\br$.

\item  [(vi)]
For arbitrary isomorphisms $\rho : \sfS'\to \sfD'$
  and $\sigma : \sfS''\to \sfD''$ of multicyclically ordered sets and
  $u$, $v$ as in~(\ref{eq:7}), one has the equality
\[
\nsP\bl\rho|_{S'\setminus \{u\}}\cup\sigma|_{S''\setminus
  \{v\}}\br
\ooo{u}{v} =
\ooo{\rho(u)}{\sigma(v)} \ \big(\nsP\(\rho\)\ot\nsP\(\sigma\)\big)
\]
of maps
$\nsP\(\sfS';g'\)\otimes \nsP\(\sfS'';g''\) \to
\nsP\bl\sfD'\cup \sfD'' 
\setminus  \{\rho(u),\sigma(v)\};g'+g''\br$.

\item  [(vii)]
For $\sfS$, $u$, $v$ as in~(\ref{eq:8}) and an isomorphism $\rho :
\sfS \to \sfD$ of multicyclically ordered sets, one has the equality
\[
\nsP\bl\rho|_{\sfD \setminus \{\rho(u),\rho(v)\}}\br \ \xi_{ab} = 
\xi_{\rho(u)\rho(v)}\nsP\(\rho\)
\]
of maps
$\nsP\(\sfS;g\) \to \nsP\bl\sfS \setminus \{\rho(u),\rho(v)\};\ g+1\br$.
\end{enumerate} 
\end{definition}

\begin{remark}
  While $\xi_{uv} =\xi_{vu}$, the behavior of the $\ooo uv$-operation
  under the interchange $u \leftrightarrow v$ is given by
  axiom~(i). Axioms (ii)--(v) are interchange rules between $\ooo
  uv$'s and the contractions, while the remaining two axioms describe
  how the structure operations behave under automorphisms. In (vi) and
  (vii) one sees the restrictions of automorphisms of multicyclically
  ordered sets. It is clear that they are automorphisms of the
  corresponding multicyclically ordered subsets.
\end{remark}

\begin{example}
\label{zas_mne_poboliva_v_krku}
With the definition of \ns\ modular operads given above, both sides
of~(\ref{za_chvili_prijde_rodina}) belong to the {\em same\/} space,
namely to $\nsP\bl[XY],[Z];g'\sqplus g'' \sqplus 1\br$. The
problem risen in Example~\ref{dead_end} is thus resolved by cutting
the two rightmost pancakes in Figure~\ref{fig:quandary} along the
dotted lines.
\end{example}

\begin{definition}
A \ns\ modular operad $\nsP$ is {\em geometric\/}, if its underlying
\ns\ modular module is geometric.
\end{definition}

Notice that the $\ooo uv$-operations always preserve the geometric
genus~(\ref{eq:Jaruska}). The contractions $\xi_{uv}$ preserve
it if $u,v$ in~(\ref{eq:8}) belong to the same component of the
multicyclically ordered set~$\sfS$, and raise it by $1$ if they belong
to different components of $\sfS$.
Therefore each \ns\ modular operad $\nsP$ contains a maximal
geometric suboperad. 

From this point on, we assume that all \ns\ modular operads are {\bf
  geometric}.  With this assumption, the only nontrivial components of
$\nsP$ in (operadic) genus $0$ are $\nsP\(C;0\)$, where $C$ is a
cyclically ordered set, i.e.\ a multicyclically ordered set with one
component. It is simple to show that the collection
\[
\nsBox\, \nsP : = \big\{\nsP\(C;0\); \ \hbox{$C$ is cyclically ordered}\big\} 
\]
together with $\ooo uv$ operations~(\ref{eq:7}) is a \ns\ cyclic
operad. So we have the forgetful functor
\begin{equation}
\label{eq:21}
\nsBox: \NsModOp \to \NsCycOp.
\end{equation}
In Section~\ref{envelopes} 
we construct its left adjoint $\nsMod :
\NsCycOp \to \NsModOp$.  

One also has the forgetful functor (the {\em desymmetrization\/}) $\Des : \ModOp \to \NsModOp$ given~by
\[
\Des(\oP)\(\sfS;g\) := \oP\(S\), 
\]
where $S$ is the underlying set of the multicyclically ordered set
$\sfS$. It has the left adjoint  $\Sym :  \NsModOp \to \ModOp$ given by 
\begin{equation}
\label{eq:20}
\Sym(\nsP)\(S;g\) := \coprod\nsP\(\sfS;g\),
\end{equation}
where the coproduct runs over all multicyclically ordered sets whose
underlying set equals~$S$. Notice that the geometricity guarantees
that the coproduct in~(\ref{eq:20}) is finite. We call $\Sym(\nsP)$
the {\em symmetrization\/} of the \nsm\ operad $\nsP$.

\begin{remark}
\label{why_geometric}
  Assuming the geometricity, the category of \ns\ cyclic operads is
  isomorphic to the full subcategory of \ns\ modular operads $\nsP$
  such that $\nsP\(\sfS;g\)= 0$ for $g \geq 1$. Without the
  geometricity assumption, this natural property that obviously holds
  for ordinary modular operads, will not be true. A `geometric'
  explanation of the geometricity will be given in
  Remark~\ref{je_mi_trochu_lip}.
\end{remark}

\begin{example}
\label{Jaruska_ma_obsazeno}
Assume that the basic monoidal category is the cartesian category $\Set$ of
sets and let $*$ be a fixed one-point set. Then one has the {\em
  terminal\/} \ns\ modular operad  $\underline*_M$ with 
\[
\underline*_M\(\sfS;g\) := * \ \hbox{ for each geometric } (\sfS,g) \in \MultCyc
\times \bbN,
\] 
with all structure operations the unique maps $* \to *$ or $*\times *
\to *$.
\end{example}

\section{Un-biased definition of non-$\Sigma$ modular operads.}

We give an alternative definition of \ns\ modular operads as algebras
over a certain monad of decorated graphs representing their pasting
schemes, thus extending the table
in~\cite[Figure~14]{markl:handbook}. This way of defining various
types of operads is standard, see
e.g.~\cite[\S2.20]{getzler-kapranov:CompM98}, \cite[II.1.12,
II.5.3]{markl-shnider-stasheff:book} or
\cite[Theorem~40]{markl:handbook}, so we only emphasize the particular
features of the \ns\ modular case.  We start by recalling a definition
of graphs suggested by Kontsevich and
Manin~\cite{kontsevich-manin:preprint} commonly used in operad
theory. More refined notions of graphs already exist, see
e.g.~\cite[Part~4]{batanin13:_homot}, but we will not need them here.

\begin{definition}
\label{graph-def}
A {\em graph\/} $\Gamma$ is a finite set $\Flag(\Gamma)$
(whose elements are called {\em flags\/} or {\em half-edges\/}) 
together with an involution 
$\sigma$ and a partition $\lambda$. 

The {\em vertices\/} $\Vert(\Gamma)$
of a graph $\Gamma$ are the blocks of the partition
$\lambda$. The {\em edges\/}
$\edge(\Gamma)$ are pairs of flags forming a two-cycle of $\sigma$
relative to the decomposition of a permutation into disjoint cycles. The
{\em legs\/} $\leg(\Gamma)$ are the fixed points of $\sigma$.
\end{definition}

We  denote by $\Leg(v)$ the flags belonging to the
block $v$ or, in common speech, half-edges adjacent to the vertex
$v$. The cardinality of $\Leg(v)$ is the {\em valency\/} of $v$.  
We say that two flags $x,y \in \Flag(\Gamma)$ {\em meet\/} if
they belong to the same block of the partition $\lambda$. In plain
language, this means that they share a common vertex. 


\begin{definition}
  A {\em \ns\ modular graph\/} $\Gamma$ is a connected graph as above that has,
  moreover, 
the following
local structure at each vertex $v\in \Vert(\Gamma)$:
\begin{itemize}
\item[(i)] 
a multicyclic order of the set $\Leg(v)$ of half-edges adjacent to
 $v$ and 
\item[(ii)]
a genus $g_v \in \mathbb N$.
\end{itemize}
\end{definition}

\begin{remark}
Non-$\Sigma$ modular graphs satisfying moreover a stability
condition already appeared under the name {\em stable graphs\/}
in~\cite[Appendix~B]{Airy}, or as {\em stable ribbon graphs\/}
in~\cite[Section~8]{barannikov}, in connection with compactifications
of moduli spaces of Riemann surfaces, cf.\
also~\cite[Section~1]{penner:JDG88}. The relation between stable
ribbon graphs and quotients of ribbon graphs was discussed
in~\cite[Section~9]{Looijenga}. The local multicyclic structure
of graphs dual to arc families was also explicitly recognized 
in~\cite[Appendix~A.1]{kaufmann:dim_vs_genus}, while the global one is
apparent at the set of marked points of 
windowed surfaces~\cite[Section~1]{kaufmann-penner:NP06}.
\end{remark}

We will denote by $\sfLeg(v)$ the set $\Leg(v)$ with the given
multicyclic order and by $b_v$ the number of cyclically ordered subsets
in the corresponding decomposition. We say that $\Gamma$ is
{\em geometric\/} if at each $v \in \Vert(\Gamma)$, 
\[
\textstyle G_v := \frac12( g_v+1 - b_v) \in \bbN.
\]

Crucially, the local structure of a \ns\
modular graph induces the same kind of structure on its external legs:

\begin{proposition}
\label{Jarka_nevola}
The set of legs of a \ns\ modular graph has an induced
multicyclic order.
\end{proposition}

\begin{proof}
By an oriented edge cycle in $\Gamma$ we understand a sequence
\[
(a_1,b_1), (a_2,b_2), \ldots, (a_s,b_s),
\]
where $a_i,b_i$ are half-edges such that $\sigma(a_i) = b_i$ for $1
\leq i \leq s$. So, if
$a_i \not= b_i$, $(a_i,b_i)$ is an oriented edge, if $a_i = b_i$ it is
a leg of $\Gamma$. We require that $a_i$ is the immediate
successor of $b_{i-1}$, $1 < i \leq s$, and that $a_1$ is the immediate
successor of $b_{s}$ in the cyclically ordered set to which these
elements belong.

If this cyclically ordered set 
consists of $b_{i-1}$ resp.\ of $b_s$ only, then of
course $a_i = b_{i-1}$ resp.\ $a_1 = b_s$ so the cycle runs back along the
same half-edge. We also assume that each ordered couple $(a_i,b_i)$ occurs
exactly once so that the cycle does  not not run twice along itself. 
This does not exclude that $(a_i,b_i) = (b_j,a_j)$, i.e.\ that the
cycle runs twice along the same edge, but each time in different direction. 

Let  $
\{X_1,\ldots,X_{b}\}$ be the set of oriented edge cycles of
$\Gamma$ and 
\begin{equation}
\label{zasil_jsem_polstarky}
C_i := \big\{ e \in \Leg(\Gamma);\ (e,e) \in X_i,\ 1 \leq i \leq
b\big\}.
\end{equation}
Then $\Rada C1b$ is the required disjoint decomposition of $\Leg(\Gamma)$ and each
individual $C_i$ is cyclically ordered by the cyclic orientation of
the corresponding edge cycle.
\end{proof}

We will denote by $\sfLeg(\Gamma)$ the set $\Leg(\Gamma)$ with the
multicyclic order of Proposition~\ref{Jarka_nevola} and by $b(\Gamma)$
the number of cyclically ordered sets in the decomposition of $\sfLeg(\Gamma)$. The
(operadic) genus of a \ns\ modular graph is defined by the usual formula
\cite[(II.5.28)]{markl-shnider-stasheff:book}
\[
g(\Gamma) := b_1(\Gamma)+\sum_{v\in \Vert(\Gamma)} g(v),
\]
where $b_1(\Gamma)$ is the first Betti number of $\Gamma$, i.e.~the number of independent
circuits of $\Gamma$. We leave as an exercise to prove

\begin{proposition}
 If $\Gamma$ is geometric then 
 $\big(\sfLeg(\Gamma),g\big) \in \MultCyc \times \bbN$ is geometric,
 too, i.e.\  
\[
\textstyle
G(\Gamma) := \frac12  \big(g(\Gamma) + b(\Gamma) -1\big) \in \bbN.
\]
\end{proposition}

A {\em morphism\/} $f:\Gamma_0 \to \Gamma_1$ of graphs
is given
by a permutation of 
vertices, followed by a contraction of 
some edges of the graph $\Gamma_0$, leaving the
legs untouched; a precise definition can be found in
\cite[II.5.3]{markl-shnider-stasheff:book}. Assume that $\Gamma_0$
and $\Gamma_1$ bear a \ns\ modular structure. It is simple to see that  
the \ns\ modular structure of $\Gamma_0$ induces via
$f$ a \ns\ modular structure on $\Gamma_1$. We say that
$f:\Gamma_0 \to \Gamma_1$ is a {\em morphism of \ns\ modular graphs\/} if this
induced structure on $\Gamma_1$ coincides with the given one.

For a multicyclically ordered set $\sfS$, let 
$\Gammacat g{\sfS}$  be the category whose objects are 
pairs $(\Gamma,\rho)$
consisting of a \ns\ modular graph $\Gamma$
of genus $g$ and an isomorphism $\rho : \sfLeg(\Gamma) \to \sfS$ of
multicyclically ordered sets. Morphisms of $\Gammacat gS$ are
morphisms 
as above preserving the labelling of the legs. The category $\Gammacat g{\sfS}$
has a terminal object $\star_{\sfS,g}$,
the `\ns\ modular corolla' with no edges, one vertex
$v$ of genus $g$ and legs labeled by $\sfS$.

For a \ns\ modular module $E$  and a 
\ns\ modular graph $\Gamma$, one forms the unordered
product~\cite[Definition II.1.58]{markl-shnider-stasheff:book}
\begin{equation} 
\label{Cosmo}
{{E}}(\Gamma): = \bigotimes_{v\in\Vert(\Gamma)} 
{{E}}\(\sfLeg(v);g_v\) .
\end{equation}
Let $\Iso\Gammacat g{\sfS}$ denote the subcategory of isomorphisms in
$\Gammacat g{\sfS}$.  
The correspondence $\Gamma \mapsto
E(\Gamma)$ extends to a functor from the category $\Iso\Gammacat g{\sfS}$ to
$\ttM$.  We define the endofunctor $\sfModtriple : \nsModMod \to \nsModMod$ on the
category of \ns\ modular modules as the colimit
\[
\sfModtriple {(E)}\(\sfS;g\) := 
\colim{{\Gamma\in\Iso\Gammacat gS}}{E}(\Gamma). 
\]
For each \nsm\ graph $\Gamma \in \Gammacat g{\sfS}$ one
has the coprojection
\begin{equation}
\label{eq:23}
\iota_\Gamma:
E(\Gamma)\ \lra\ \sfModtriple {(E)}\(\sfS,g\).
\end{equation}
In particular, for
the corolla   $\star_{\sfS,g}$, one gets the morphism
\begin{equation}
\label{eq:24}
\iota_{\star_{\sfS,g}} 
: E(\sfS;g) = E(\star_{\sfS,s}) \ \lra \ \sfModtriple {(E)}\(\sfS;g\).
\end{equation}

Given a \nsm\ module $E$, the second iterate $(\Free\circ \Free) (E)$ is
a colimit of \nsm\ graphs whose vertices are decorated by \nsm\ graphs
decorated by $E$, i.e.\ a colimit of `nested' graphs. Forgetting the
nests gives rise to a natural transformation
\[
\mu :\ \Free\circ \Free \ \lra\ \Free\ \hbox { (the multiplication)}    
\]
while morphisms~(\ref{eq:24}) form a natural transformation
\[
\nu:\id\   \lra \  \Free\ \hbox { (the unit).} 
\]
Precisely as in~\cite[\S2.17]{getzler-kapranov:CompM98} 
or in the proof of~\cite[Theorem~II.5.10]{markl-shnider-stasheff:book} one 
shows that  $\sfModtriple = (\sfModtriple,\mu,\nu)$ is a
monad on the category $\nsModMod$. We have the
following theorem/definition.

\begin{theorem}
\label{sec:un-biased-definition}
Non-$\Sigma$  modular operads are algebras for the monad $\Free = (\Free,\mu,\nu)$.
\end{theorem}

\begin{proof}
A straightforward modification of the proof of 
\cite[Theorem~II.5.41]{markl-shnider-stasheff:book}.
\end{proof}

A \nsm\ operad is thus a \nsm\ module $\nsP$ equipped with a
morphism
\begin{equation}
\label{eq:Jarka_by_uz_mela_prijet!}
\alpha : \Free ( \nsP) \ \lra \ \nsP
\end{equation}
of \nsm\ modules having the usual properties
\cite[Definition~II.1.103]{markl-shnider-stasheff:book}. The following
claim expresses a standard property of algebras over a monad.

\begin{proposition}
The multiplication $\mu : (\Free \circ \Free)( E) \to \Free (E)$ makes 
$\Free(E)$ an algebra for the monad $\Free$. It is the {\em
  free \ns\ modular operad\/} on the \ns\ modular module $E$.
\end{proposition}

Thus $\Free(-)$ interpreted as a functor 
$\nsModMod \to \NsModOp$ is the left adjoint to the
obvious forgetful functor $\nsF :\NsModOp \to \nsModMod$. 

\begin{remark}
\label{K_JArce_zalit_kyticky}
The biased structure operations of the free operad $\Free(E)$ are
induced by the grafting of the underlying graphs. For graphs
$\Gamma'$, $\Gamma''$ with legs $u \in \Leg(\Gamma')$, $v \in
\Leg(\Gamma'')$, one has the graph $\Gamma' \ooo uv \Gamma''$ obtained
by grafting the free end of the half-edge $u$ to the free end of 
$v$. Formally, $\Gamma' \ooo uv \Gamma''$ is defined by 
\[
\Flag(\Gamma' \ooo uv \Gamma'') : = \Flag(\Gamma') \cup
\Flag(\Gamma''),
\]
the partition of $\Flag(\Gamma' \ooo uv \Gamma'')$ being the union of the
partitions of $\Flag(\Gamma')$ and $\Flag(\Gamma'')$, the
involution $\sigma$ on $\Flag(\Gamma_1 \ooo uv \Gamma_2)$ agreeing with
the involution $\sigma'$ of $\Flag(\Gamma')$ on $\Flag(\Gamma')
\setminus \stt u$, with the involution $\sigma''$ of $\Flag(\Gamma'')$
on $\Flag(\Gamma'') \setminus \stt v$, and $\sigma(u) := v$. The
contraction $\xi_{uv}(\Gamma)$ is, for $u,v \in \Flag(\Gamma)$,
defined similarly.
\end{remark}

\begin{example}
\label{zase_mne_boli_v_krku_tentokrat_jsem_to_chytl_od_Andulky}
Assume that $E$ is a geometric \nsm\ module such that $E\(\sfS;g\) = 0$
for $g \geq 1$. In other words, the only nontrivial pieces of $E$ are
$E\(C;0\)$, where $C$ is a cyclically ordered set.  The
elements of the free \nsm\ operad $\Free(E)$ are the
equivalence classes of decorated graphs whose vertices $v$ are umbels with one
blossom, i.e.\ the pancakes~(\ref{the_pancake}) with the spikes
representing the half-edges in the cyclically ordered set $\Leg(v)$.
Free operads of this form will play a central r\^ole in our
proof of~(\ref{Jarka_stale_bez_prace-b}).

A very particular case is the geometric \nsm\ module $\underline{*}$ in
$\Set$  defined by
\[
\underline{*}\(\sfS;g\) := 
\begin{cases}
{*}&\hbox{if $g=0$ and}
\\
\emptyset& \hbox{otherwise.}
\end{cases}
\]
In Section~\ref{lazarev} we visualize the generators of $\Free(\underline{*})$ via
cog wheels~(\ref{Pozitri_se_sem_nastehuje_Andulka.}).  
\end{example}

For a \nsm\ operad $\nsP$ and $\Gamma \in \Gammacat g{\sfS}$ we will call
the composition
\begin{equation}
  \label{eq:Jarusko_prijed!}
\alpha_\Gamma : \nsP(\Gamma) \stackrel{\iota_\Gamma}{\lra}       
(\Free\ \nsP)  \(\sfS;g\) \stackrel{\alpha}{\lra}  \nsP\(\sfS;g\) 
\end{equation}
of~(\ref{eq:Jarka_by_uz_mela_prijet!}) with~(\ref{eq:23}) the {\em
  contraction\/} along the graph $\Gamma$.
The following analog of
\cite[Theorem~II.5.42]{markl-shnider-stasheff:book} 
claims that the contractions are part of a functor:

\begin{theorem}
\label{Opicak_Fuk}
A geometric \nsm\ module $\nsP$ is a \nsm\ operad if and only if the
correspondence $\Gamma \mapsto \nsP(\Gamma)$ is, for each geometric
$(\sfS,g) \in \MultCyc\! \times\! \bbN$, an object part of a functor
$\alpha : \Gammacat{g}{\sfS} \to \ttM$
extending~(\ref{eq:Jarusko_prijed!}). By this we mean that $\alpha(f)
= \alpha_\Gamma$ for the unique morphism $f: \Gamma \to
\star_{\sfS,g}$.
\end{theorem}

\begin{proof}
A simple modification of the proof of \cite[Theorem~II.5.42]{markl-shnider-stasheff:book}.
\end{proof}

\section{Modular envelopes}
\label{envelopes}

Modular envelopes of (ordinary) cyclic operads were introduced under
the name modular operadic completions by the author
in~\cite[Definition~2]{markl:la}.  The {\em modular envelope\/}
functor $\Mod: \CycOp \to \ModOp$ is the left adjoint to the obvious
forgetful functor $\Box: \ModOp \to \CycOp$. Analogously we define the
modular envelope of a \ns\ cyclic operad via the left adjoint to the
\ns\ version $\nsBox$ of the forgetful functor considered in
Section~\ref{sec:bias-defin-non}.  We of course need to prove that
this left adjoint exists:

\begin{proposition}
The forgetful functor~(\ref{eq:21}) has a left adjoint $\nsMod:
\NsCycOp \to \NsModOp$. 
\end{proposition}
  
\begin{proof}
  The proof will be transparent if we assume that the objects of
  the basic category $\ttM$ have elements. Then we take the
  free \ns\ modular operad $\Free(\nsF\, \nsP)$ generated by the
  \ns\ cyclic collection $\nsF \,\nsP$ placed in the operadic genus
  $0$ 
and define
  $\nsMod(\nsP)$ as the quotient
\begin{equation}
\label{uplne_na_mne_masti}
\nsMod(\nsP) := \Free(\nsF\,\nsP)/{\mathcal I}
\end{equation}
of $\Free(\nsF\,\nsP)$ by the operadic ideal ${\mathcal I}$
generated by
\[
x {}_u {\circ^\nsP \hskip -.3em}_v y = x {}_u {\circ^\Free \hskip -.3em}_v y,
\] 
where $ {}_u {\circ^\nsP \hskip -.3em}_v$ resp.\ $ {}_u {\circ^\Free
  \hskip -.3em}_v$ are the $\ooo uv$-operations in $\nsP$ resp.\
$\Free$, $x \in \nsP(C')$, $y \in \nsP(C'')$, $u\in C'$ and $v \in
C''$ for some disjoint cyclically ordered sets $C'$, $C''$. 

If objects of $\ttM$ do not have elements, we replace the
quotient~(\ref{uplne_na_mne_masti}) by an obvious colimit. It is clear
that~(\ref{uplne_na_mne_masti}) defines a left adjoint
to~(\ref{eq:21}).
\end{proof}

\begin{definition}
We call $\nsMod(\nsP)$ the {\em \ns\ modular
  envelope\/} of the \ns\ cyclic operad $\nsP$.
\end{definition}

Informally, $\nsMod(\nsP)$ is obtained by adding
to $\nsP$ the results of contractions, splinting the cyclically ordered
groups of inputs if the contraction takes place
within the same group. This process is nicely visible at
Doubek's  construction of the modular envelope of the operad for
associative algebras~\cite{doubek:modass}.

If the basic category $\ttM$ is \Set, the category of cyclic
(resp.\ \ns\ cyclic, resp.\ modular, resp.\ \nsm) operads has a
terminal object $*_C$ (resp.\ $\underline*_C$, resp.\
$*_M$ resp.\
$\underline*_M$) consisting of a chosen one-point set $*$ in each
arity; the terminal \nsm\ operad $\underline*_M$ has already been
mentioned in Example~\ref{Jaruska_ma_obsazeno}. The main result of this
section is:

\begin{theorem}
\label{sec:modular-envelopes-6}
The modular envelope of the terminal operad in the category of cyclic
(resp.\ \ns\ cyclic) \Set-operads
is the terminal modular (resp.\ the
terminal \ns\ modular) \Set-operad, in formulas:
\begin{subequations}
\begin{align}
\label{Jarka_stale_bez_prace-a}
\Mod(*_C) & \cong *_M \ \hbox { and } 
\\ 
\label{Jarka_stale_bez_prace-b}
\nsMod(\underline*_C) & \cong \underline*_M.
\end{align}
\end{subequations}
\end{theorem}

The cyclic operad $\Com$ for commutative associative algebras is the
linear span of the terminal cyclic $\Set$-operad, that is 
$\Com = \Span(*_C)$. 
The following immediate corollary of
Theorem~\ref{sec:modular-envelopes-6} 
was stated without proof in
\cite[page~382]{markl:la}.  

\begin{theorem}[\cite{markl:la}]
\label{sec:modular-envelopes}
The modular envelope $\Mod(\Com)$ of the cyclic operad for commutative
associative algebras is the linear span of the terminal
modular set-operad, i.e. 
\[
\Mod(\Com)\(S;g\) = \bfk, \ (S,g) \in \Fin\! \times\! \bbN,
\]
the maps $\Mod(\Com)\(\sigma\)$ induced by morphisms in $\Fin$
are the identities,
all $\ooo uv$-operations are the canonical
isomorphisms $\bfk \ot \bfk  \rediso  \bfk$ and all contractions $\xi_{uv}$ are the
identities.
\end{theorem}

\begin{proof}
Both $\Span(-)$ and $\Mod(-)$ are the left adjoints to forgetful functors
that commute with each other, so 
\[
\Span\big(\Mod({\mathcal S})\big) \cong \Mod\big(\Span({\mathcal S})\big)
\]
for each cyclic operad ${\mathcal S}$ in $\Set$.
\end{proof}

Since the \ns\ cyclic operad $\nsAss$ is the linear span of the
terminal  \ns\ cyclic  $\Set$-operad $\underline *_C$, 
we likewise obtain from Theorem~\ref{sec:modular-envelopes-6}:

\begin{theorem}
\label{sec:modular-envelopes-1}
The \ns\ modular envelope $\nsMod(\nsAss)$ of the \ns\ cyclic operad
$\nsAss$ for associative algebras is the linear span of the terminal \ns\ modular operad in the
category of sets. Explicitly,
\[
\nsMod(\nsAss)\(\sfS;g\) = \bfk
\]
for each multicyclically ordered set $\sfS$
  and
$g \in \bbN$. All structure operations are either the identities
of $\bfk$ or the canonical isomorphisms $\bfk \ot \bfk \rediso \bfk$.
\end{theorem}

Since the symmetrization~(\ref{eq:20}) clearly commutes with the \ns\
modular envelope functor, Theorem~\ref{sec:modular-envelopes-1}
implies the isomorphisms
\begin{equation}
\label{Jaruska_dalsi_tyden_na_chalupe.}
\Mod(\Ass) \cong \Sym\big(\Span(\underline*_M)\big)
\end{equation} 
proved in \cite{chuang-lazarev:dual-fey,doubek:modass}
though not expressed in this form there.

Let us start proving the isomorphisms~(\ref{Jarka_stale_bez_prace-a})
and~(\ref{Jarka_stale_bez_prace-b}) of
Theorem~\ref{sec:modular-envelopes-6}.  From here till the end of this
section the basic category will be the category of sets.

\begin{proof}[Proof of~(\ref{Jarka_stale_bez_prace-a})]
  It will be a warm-up for the proof
  of~(\ref{Jarka_stale_bez_prace-b}) given below.  The modular
  envelope $\Mod(*_C)$ is characterized by the adjunction
\begin{equation}
\label{Zase_podleham}
\ModOp\big(\Mod(*_C),\oP\big) \cong \CycOp(*_C,\Box\, \oP)
\end{equation}
that must hold for each modular $\Set$-operad $\calP$.
It is clear that there is a
one-to-one correspondence between morphisms in $\CycOp(*_C,\Box\, \oP)$
and families
\begin{equation}
\label{boli_mne_v_krku}
\varsigma(S) \in \oP\(S;0\),\ S \in \Fin,
\end{equation}
such that
\begin{equation}
\label{eq:12}
\oP\(\sigma\)\big(\varsigma(S)\big) = \varsigma(D) \hbox { and }
\varsigma(S') \ooo uv \varsigma(S'') = \varsigma (S' \cup S'' 
\setminus \{u,v\})
\end{equation}
for each $S',S'', u,v,\sigma$ for which the above expressions make sense. 

The theorem will obviously be proved if we exhibit a
one-to-one correspondence between morphisms   $*_C \to \Box\,\oP$
represented by~(\ref{boli_mne_v_krku}), and families 
\begin{equation}
\label{1boli_mne_v_krku}
\varpi(S;g) \in \oP\(S;g\),\ S \in \Fin\! \times\! \bbN,
\end{equation}
such that  $\varpi(S;0) = \varsigma(S)$ for each $S \in \Fin$, and    
\begin{subequations}
\begin{align}
\label{eq:9}
\oP\(\sigma\)\big(\varpi(S;g)\big)& = \varpi(D;g),\ \sigma : S \to D \in \Fin,
\\
\label{eq:2}
\varpi(S';g') \ooo uv \varpi(S'';g'')& = \varpi \big(S' \cup S'' 
\setminus \{u,v\};g' + g''\big), \ \hbox{ and }
\\
\label{eq:10}
\xi_{uv}\varpi(S;g) &= \varpi\big(S \setminus \{u,v\}; g-2\big),
\end{align}
\end{subequations}
whenever the above objects are defined. In the light
of~(\ref{Zase_podleham}) this is the same as to show that 
each family~(\ref{boli_mne_v_krku}) uniquely
determines a family~(\ref{1boli_mne_v_krku}) with $\varsigma(S;0) =
\varsigma(S)$ for each $S \in \Fin$.

Let $\Gamma$ be a connected graph  of genus $g$ with trivial local genera
$g_v$ at each vertex, and $\Leg(\Gamma) = S$. Decorate the
vertices of $\Gamma$ by~(\ref{boli_mne_v_krku}) and denote the result
by 
\[
\varsigma(\Gamma) = \bigotimes_{v \in \Vert(\Gamma)}
\varsigma\big(\Leg(v)\big) \in \Freesym(\oP)\(S;g\),
\]
where $\Freesym(\oP)$ is the free modular
operad~\cite[Section~II.5.3]{markl-shnider-stasheff:book} on the
modular module $F\oP$. Formally correct notation would
  therefore be $\Freesym(F \oP)$ but we want to save space.
The composition
(contraction) \hbox{$c:\Freesym(\oP)\(S;g\) \to \oP\(S;g\)$}
along the graph $\Gamma$ determines
\begin{equation}
\label{eq:13}
\varpi_\Gamma := c\big(\varsigma(\Gamma)\big)  \in \oP\(S;g\).
\end{equation}
Assume we proved, for $\Gamma$, $S$ and $g$ as above, that 
\begin{equation}
\label{eq:11}
\hbox{$\varpi_\Gamma$ depends only on the type of $\Gamma$, i.e. on
$S$ and $g$, not on the concrete $\Gamma$.}
\end{equation}
We claim that then $\varpi(S;g) := \varpi_\Gamma$ is the
requisite extension of~(\ref{1boli_mne_v_krku}).

Indeed, to establish~(\ref{eq:9}), denote by $\Gamma'$ the graph
$\Gamma$ with the legs relabeled according to
$\sigma$. Then~(\ref{eq:11}) implies that
$\oP\(\sigma\) \varpi_\Gamma = \varpi_{\Gamma'}$. As for~(\ref{eq:2}),  
assume that $\varpi(S';g') =
\varpi_{\Gamma'}$ and  $\varpi(S'';g'') =
\varpi_{\Gamma''}$. 
Then $\varpi(S';g') \ooo uv
\varpi(S'';g'')$ equals $\varpi_{\Gamma' \ooo uv \Gamma''}$ with
$\Gamma' \ooo uv \Gamma''$ the grafting recalled in Remark \ref{K_JArce_zalit_kyticky},
which in turn equals $\varpi (S' \cup S'' 
\setminus \{u,v\};g'\! + \! g'')$. This
proves~(\ref{eq:2}). Property~(\ref{eq:10}) can be
discussed similarly. 

Choose a maximal subtree $T$ of $\Gamma$ and denote by $K := \Gamma/T$
the result of shrinking $T \subset K$ into a corolla. Notice that $K$
has only one vertex. By the
associativity~\cite[Theorem~II.5.42]{markl-shnider-stasheff:book}
of the contractions, $\varpi_\Gamma =\varpi_K$, it is therefore enough
to prove~(\ref{eq:11}) for graphs $K$ with one vertex. Graphically,
such a $K$ is a non-planar `tick' with $g$ `bellies' as in
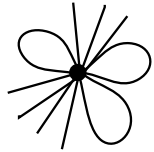
\begin{center}
\psscalebox{.3.3} 
{
\begin{pspicture}(0,-2.1025)(7.775051,2.6)
\psbezier[linecolor=black, linewidth=0.102](4.385,0.2225)(4.698754,-0.7461785)(4.7731714,-2.3635066)(5.42371,-2.8775)(6.074248,-3.3914936)(6.9987535,-2.5634198)(6.685,-1.5947413)(6.3712463,-0.62606287)(5.3323126,-0.5052866)(4.681774,0.008706897)
\psbezier[linecolor=black, linewidth=0.102](4.485,0.1225)(5.482268,0.048628315)(6.2240734,-0.6862293)(7.085,-0.1775)(7.9459267,0.3312293)(7.6822677,1.3486283)(6.685,1.4225)(5.687732,1.4963716)(5.2459264,0.5312293)(4.385,0.0225)
\psbezier[linecolor=black, linewidth=0.102](4.585,0.0225)(3.991801,0.82755584)(4.040779,1.6284142)(3.085,1.9225)(2.129221,2.2165859)(1.391801,1.6275558)(1.985,0.8225)(2.5781991,0.017444164)(3.529221,0.31658584)(4.485,0.0225)
\psline[linecolor=black, linewidth=0.102](4.485,0.1225)(4.185,3.2225)(4.185,3.2225)
\psline[linecolor=black, linewidth=0.102](4.485,0.0225)(5.585,3.0225)(5.585,3.1225)
\psline[linecolor=black, linewidth=0.102](4.585,0.0225)(6.585,2.6225)(6.585,2.6225)
\psline[linecolor=black, linewidth=0.102](4.485,0.1225)(1.285,-0.7775)(1.285,-0.7775)
\psline[linecolor=black, linewidth=0.102](4.585,0.0225)(1.785,-1.8775)(1.785,-1.7775)
\psline[linecolor=black, linewidth=0.102](4.485,0.2225)(3.685,-3.2775)(3.685,-3.2775)
\psline[linecolor=black, linewidth=0.102](4.485,0.2225)(2.585,-2.5775)(2.585,-2.5775)
\rput{-1.2351013}(-0.0016216917,0.094546765){\psdots[linecolor=black,
  dotsize=0.7671875](4.385,0.1225)}
\end{pspicture}
}
\end{center}
where $g=3$. We prove~(\ref{eq:11}) by
induction on the genus $g$.

If $g = 0$, $K$ is the corolla with $\Leg(K) = S$ whose unique vertex
is decorated by $\varpi_T$ which equals $\varsigma(S)$, because the
contraction in~(\ref{eq:13}) uses only the cyclic part $\Box\,\oP$ of
the operad~$\oP$. Therefore $\varpi_T$ does not depend on the choice
of $T$ and is determined by a given map $*_C \to \Box\, \oP$.

Assume we have proved~(\ref{eq:11}) for all $g'< g$. Let $K'$ be a
tick obtained by removing one belly of $K$. Then clearly $\Leg(K') = S
\cup \{u,v\}$ for some $u$ and $v$, and
\[
\varpi_K =  \xi_{uv} (\varpi_{K'}).
\]
By the induction assumption, $\varpi_{K'}$ equals $\varpi\big( S \cup
\{u,v\};g-1\big)$ and does not depend on the concrete form of $K'$, so
\[
\varpi_K =  \xi_{uv}\big( \varpi( S \cup
\{u,v\};g-1)\big)
\]
depends only on the finite set $S$ and the genus $g$.
\end{proof}

Let us formulate a useful

\begin{definition}
\label{internal}
A family~(\ref{boli_mne_v_krku}) satisfying~(\ref{eq:12}) is an
  {\em internal operad\/} in the cyclic operad $\Box\, \oP$. Similarly~(\ref{1boli_mne_v_krku})
  satisfying~(\ref{eq:9})--(\ref{eq:10}) is an {\em internal operad\/} in the
  modular operad~$\oP$. These notions have obvious \ns\ analogs.
\end{definition}

The proof of~(\ref{Jarka_stale_bez_prace-b}) occupies the rest
of this section.
 Since it follows the scheme of the proof of
  Theorem~\ref{sec:modular-envelopes}, we only emphasize the
  differences. We must to show that each
  internal \ns\ cyclic operad
\begin{equation}
\label{boli_mne_v_krku1}
\varsigma(C) \in \nsP\(C;0\),\ C \in \Cyc,
\end{equation}
in $\nsBox\, \nsP$ uniquely
extends to an internal \ns\ modular operad 
\[
\varpi(\sfS;g) \in \nsP\(\sfS,g\),\ \sfS \in \MultCyc \! \times\! \bbN
\]
in $\nsP$ 
such that $\varpi(C;0) =  \varsigma(C)$ for $C \in \Cyc \subset
\MultCyc$.

As in the proof of Theorem~\ref{sec:modular-envelopes}, for a \ns\
modular graph $\Gamma$ with $\Leg(\Gamma) = \sfS$ whose local genera
vanish we define $\varpi_\Gamma \in \nsP\(\sfS;g\)$ using the
contractions~(\ref{eq:Jarusko_prijed!}) along $\Gamma$ in $\nsP$. We then need to prove an analog
of~(\ref{eq:11}):
\begin{equation}
\label{eq:1bis1}
\begin{array}{cc}
\hbox{$\varpi_\Gamma$ depends only on the type of $\Gamma$, i.e.\ on
  the multicyclically }
\\
\hbox{ordered set 
$\sfS = \Leg(\Gamma)$ and the genus $g = g(\Gamma)$.}
\end{array}
\end{equation}

By choosing a maximal   
subtree $T$ of $\Gamma$ we again reduce (\ref{eq:1bis1}) to graphs $K$ 
with only one vertex. This time, $K$ is not a `tick'
but a pancake
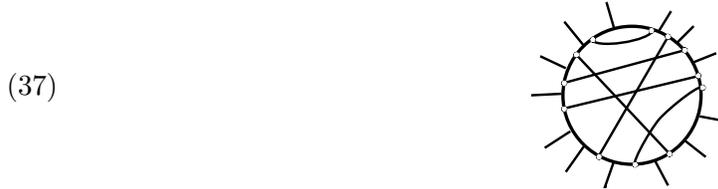
\begin{equation}
\label{Jaruska_se_uz_na_chalupe_prudi.}
\raisebox{-2.8em}{}
\psscalebox{.2.2}
{
\begin{pspicture}(0,0)(12.818973,6.)
\pscircle[linecolor=black, linewidth=0.22, dimen=outer](6.7040453,-0.050000153){4.7}
\psline[linecolor=black, linewidth=0.162](2.1040454,-0.95000017)(11.204045,1.2499999)(11.204045,1.2499999)
\psline[linecolor=black, linewidth=0.162](3.0040455,2.6499999)(9.304046,-4.05)(9.304046,-4.05)
\psline[linecolor=black, linewidth=0.162](9.104046,3.85)(4.4040456,-4.25)(4.4040456,-4.25)
\psbezier[linecolor=black, linewidth=0.162](7.1040454,-4.65)(6.27657,-5.2115016)(7.853254,-2.4092567)(8.5040455,-1.6500001)(9.154837,-0.89074355)(11.726931,1.156183)(11.304046,0.24999985)
\psbezier[linecolor=black, linewidth=0.162](4.0040455,3.7499998)(4.0040455,2.9499998)(8.104046,3.55)(8.104046,4.35)
\psline[linecolor=black, linewidth=0.162](2.1040454,0.7499998)(10.204045,2.9499998)(10.204045,2.9499998)
\psline[linecolor=black, linewidth=0.162](8.5040455,4.25)(9.304046,5.5499997)(9.304046,5.5499997)
\psline[linecolor=black, linewidth=0.162](9.704045,3.4499998)(10.804046,4.5499997)(10.804046,4.5499997)
\psline[linecolor=black, linewidth=0.162](10.804046,2.1499999)(12.304046,2.7499998)(12.304046,2.7499998)
\psline[linecolor=black, linewidth=0.162](11.204045,-1.4500002)(12.804046,-1.7500001)(12.804046,-1.7500001)
\psline[linecolor=black, linewidth=0.162](10.204045,-3.0500002)(11.604046,-4.15)(11.604046,-4.15)
\psline[linecolor=black, linewidth=0.162](8.204045,-4.3500004)(9.104046,-6.05)(9.104046,-6.05)
\psline[linecolor=black, linewidth=0.162](5.5040455,-4.4500003)(4.9040456,-6.15)(5.0040455,-6.15)
\psline[linecolor=black, linewidth=0.162](3.5040455,-3.45)(2.3040454,-5.15)(2.3040454,-5.15)
\psline[linecolor=black, linewidth=0.162](2.6040454,-2.45)(0.9040454,-3.5500002)(0.9040454,-3.5500002)
\psline[linecolor=black, linewidth=0.162](2.0040455,0.049999848)(0.00404541,-0.050000153)(0.00404541,-0.050000153)
\psline[linecolor=black, linewidth=0.162](2.3040454,1.7499999)(0.6040454,2.55)(0.6040454,2.55)
\psline[linecolor=black, linewidth=0.162](3.5040455,3.2499998)(2.2040453,4.85)(2.2040453,4.85)
\psline[linecolor=black, linewidth=0.162](5.5040455,4.45)(5.0040455,6.1499996)(5.0040455,6.1499996)
\psdots[linecolor=black, dotstyle=o, dotsize=0.4](9.204045,-3.95)
\psdots[linecolor=black, dotstyle=o, dotsize=0.4](11.404045,0.44999984)
\psdots[linecolor=black, dotstyle=o, dotsize=0.4](11.104046,1.2499999)
\psdots[linecolor=black, dotstyle=o, dotsize=0.4](10.204045,2.9499998)
\psdots[linecolor=black, dotstyle=o, dotsize=0.4](9.104046,3.85)
\psdots[linecolor=black, dotstyle=o, dotsize=0.4](8.0040455,4.25)
\psdots[linecolor=black, dotstyle=o, dotsize=0.4](4.1040454,3.6499999)
\psdots[linecolor=black, dotstyle=o, dotsize=0.4](3.0040455,2.6499999)
\psdots[linecolor=black, dotstyle=o, dotsize=0.4](2.2040453,0.7499998)
\psdots[linecolor=black, dotstyle=o, dotsize=0.4](2.2040453,-0.95000017)
\psdots[linecolor=black, dotstyle=o, dotsize=0.4](4.5040455,-4.15)
\psdots[linecolor=black, dotstyle=o, dotsize=0.4](6.9040456,-4.65)
\end{pspicture}
}
\end{equation}
with the internal ribs marking the half-edges which have been contracted.

\begin{example}
\label{sec:modular-envelopes-5}
The only  pancake $K$ with the operadic genus $g=0$ 
has $b(K) = 0$. It is a
circle with the circumference decorated by a cyclically ordered set $C
:= \Leg(K)$. In this case $\varpi_K = \varsigma(C)$, so~(\ref{eq:11}) is
satisfied trivially. 
\end{example}

\begin{example}
\label{sec:modular-envelopes-4}
There is only one type of a  pancake with $g= 1$, the left one in
Figure~\ref{fig:zitra_do_Tabora_na_Vivat_tour}.  
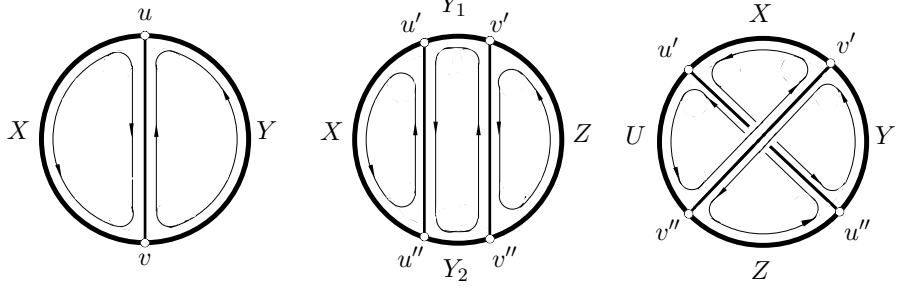
\begin{figure}[t]
  \centering
\rule{0em}{8em}
\hskip 5em
\psscalebox{.3 .3}{ 
\begin{pspicture}(5,-5)(0,0)
\pscircle[linecolor=black, linewidth=0.22, dimen=outer](5.725,-0.1775){4.7}
\pscircle[linecolor=black, linewidth=0.04, dimen=outer](5.725,-0.1775){4.1}
\pscircle[linecolor=black, linewidth=0.04, dimen=outer](4.495,3.0525){0.65}
\pscircle[linecolor=black, linewidth=0.04, dimen=outer](6.785,3.1625){0.6}
\pscircle[linecolor=black, linewidth=0.04, dimen=outer](4.625,-3.4775){0.6}
\psline[linecolor=white, linewidth=0.1](2.125,1.5225)(2.125,1.6225)(2.125,1.2225)(2.225,1.3225)(2.225,1.3225)
\pscircle[linecolor=white, linewidth=0.1, dimen=outer](3.965,2.2825){0.1}
\pscircle[linecolor=white, linewidth=0.1, dimen=outer](4.045,3.4025){0.1}
\psline[linecolor=white, linewidth=0.1](3.925,-2.4775)(3.625,-2.7775)(3.625,-2.7775)
\psline[linecolor=white, linewidth=0.1](3.825,-3.6775)(3.925,-3.7775)
\psline[linecolor=white, linewidth=0.1](6.925,3.6225)(7.225,3.5225)(7.225,3.5225)
\psline[linecolor=white, linewidth=0.1](9.125,2.0225)(9.125,1.8225)
\psline[linecolor=white, linewidth=0.1](6.125,3.5225)(5.925,3.2225)
\psline[linecolor=white, linewidth=0.2](3.825,-2.5775)(4.025,-2.6775)
\psline[linecolor=white, linewidth=0.1](9.425,-3.2775)(9.225,-3.5775)(9.225,-3.5775)
\psline[linecolor=white, linewidth=0.1](9.425,-1.6775)(9.325,-1.9775)
\psline[linecolor=white, linewidth=0.1](3.825,2.5225)(3.825,2.3225)
\psline[linecolor=black, linewidth=0.04, arrowsize=0.04cm 4.0,arrowlength=3.0,arrowinset=0.0]{<-}(9.125,2.1225)(9.325,1.7225)
\psline[linecolor=black, linewidth=0.04, arrowsize=0.04cm 4.0,arrowlength=3.0,arrowinset=0.0]{<-}(8.025,-3.5775)(7.725,-3.6775)
\psline[linecolor=black, linewidth=0.04, arrowsize=0.04cm 4.0,arrowlength=3.0,arrowinset=0.0]{<-}(2.025,-1.8775)(1.925,-1.5775)
\psline[linecolor=black, linewidth=0.04](6.225,3.1225)(6.225,-3.3775)
\pscircle[linecolor=black, linewidth=0.04, dimen=outer](6.825,-3.4775){0.6}
\psline[linecolor=white, linewidth=0.08](5.325,-4.1775)(6.025,-4.1775)
\psline[linecolor=white, linewidth=0.08](6.525,-4.0775)(6.125,-4.3775)
\psline[linecolor=white, linewidth=0.08](2.725,2.5225)(2.525,2.2225)
\psline[linecolor=white, linewidth=0.08](2.625,-2.7775)(2.925,-2.9775)
\psline[linecolor=white, linewidth=0.08](3.725,-1.9775)(3.725,-2.5775)(3.725,-2.5775)
\psline[linecolor=white, linewidth=0.08](8.425,2.7225)(8.625,2.6225)
\psline[linecolor=white, linewidth=0.08](8.825,-2.4775)(8.725,-2.2775)
\psline[linecolor=white, linewidth=0.08](8.725,-2.7775)(8.625,-2.9775)(8.625,-2.9775)
\psline[linecolor=white, linewidth=0.08](5.425,3.8225)(5.925,3.8225)
\psline[linecolor=white, linewidth=0.08](6.525,3.2225)(6.525,2.7225)
\psline[linecolor=white, linewidth=0.08](4.825,-3.1775)(4.825,-3.6775)
\psline[linecolor=white, linewidth=0.08](4.825,-3.4775)(4.725,-3.4775)
\psline[linecolor=white, linewidth=0.144](2.625,-2.6775)(2.625,-2.6775)
\psline[linecolor=white, linewidth=0.144](5.825,3.5225)(5.825,3.6225)(5.925,3.6225)
\psline[linecolor=white, linewidth=0.144](5.525,-4.0775)(5.725,-4.0775)
\psline[linecolor=white, linewidth=0.144](7.725,-2.2775)(7.625,-2.2775)
\psline[linecolor=black, linewidth=0.04, arrowsize=0.04cm 4.0,arrowlength=3.0,arrowinset=0.0]{<-}(6.2245874,0.5224997)(6.225413,0.02250034)
\psline[linecolor=black, linewidth=0.04, arrowsize=0.04cm 4.0,arrowlength=3.0,arrowinset=0.0]{<-}(5.125,-0.0775)(5.125,0.4225)
\rput(0.125,0.2225){\psscalebox{3 3}{$X$}}
\psline[linecolor=black, linewidth=0.04](5.165,3.0825)(5.165,-3.4175)
\psline[linecolor=white, linewidth=0.08](4.025,-3.7775)(4.425,-3.9775)
\psline[linecolor=white, linewidth=0.08](5.025,-2.6775)(5.025,-3.3775)
\psline[linecolor=white, linewidth=0.08](5.125,-3.0775)(5.025,-3.2775)
\psline[linecolor=white, linewidth=0.08](6.325,-3.0775)(6.325,-3.3775)
\psline[linecolor=white, linewidth=0.08](6.325,-3.1775)(6.325,-3.3775)
\psline[linecolor=white, linewidth=0.08](7.025,-3.8775)(7.925,-3.5775)(8.025,-2.4775)
\psline[linecolor=white, linewidth=0.08](7.025,-3.9775)(7.425,-3.7775)(7.725,-3.6775)(7.725,-3.6775)
\psline[linecolor=white, linewidth=0.08](7.325,-3.6775)(7.625,-3.4775)(7.625,-3.3775)
\psline[linecolor=white, linewidth=0.08](5.025,3.0225)(5.125,3.0225)(5.125,3.0225)
\psline[linecolor=white, linewidth=0.08](5.025,2.9225)(5.125,3.0225)(5.125,3.0225)
\psline[linecolor=white, linewidth=0.08](5.125,2.9225)(5.025,2.6225)
\psline[linecolor=white, linewidth=0.08](6.325,3.2225)(6.325,3.0225)(6.325,2.8225)(6.325,2.8225)
\psline[linecolor=white, linewidth=0.08](4.325,3.7225)(4.825,3.7225)(7.125,3.8225)(7.125,3.8225)
\psline[linecolor=white, linewidth=0.08](4.325,-4.0775)(5.025,-4.1775)
\psline[linecolor=white, linewidth=0.08](6.225,-4.0775)(7.025,-4.1775)(7.025,-4.1775)
\psline[linecolor=white, linewidth=0.16](3.925,3.3225)(4.025,2.6225)(3.825,3.2225)(4.125,2.5225)(4.325,2.5225)(4.325,2.3225)(4.425,2.5225)(4.525,2.4225)(4.725,2.4225)(4.825,2.5225)(5.025,2.8225)(5.025,2.8225)
\psline[linecolor=white, linewidth=0.16](3.925,3.3225)(3.925,3.3225)
\psline[linecolor=white, linewidth=0.16](4.625,3.8225)(5.525,3.9225)(6.125,3.9225)(6.525,3.8225)(6.525,3.8225)
\psline[linecolor=white, linewidth=0.16](5.025,-4.1775)(5.925,-4.2775)(6.725,-4.1775)(6.725,-4.1775)
\psline[linecolor=white, linewidth=0.16](6.325,-3.0775)(6.725,-2.8775)(7.025,-2.8775)(7.225,-3.0775)(7.325,-3.1775)(7.425,-3.3775)(7.325,-3.7775)(7.625,-3.5775)
\psline[linecolor=white, linewidth=0.16](6.325,2.9225)(6.425,2.6225)(6.625,2.6225)(6.825,2.5225)(7.025,2.6225)(7.125,2.7225)(7.225,2.8225)(7.325,3.0225)(7.425,3.1225)(7.325,3.3225)(7.225,3.5225)(7.225,3.5225)
\psline[linecolor=white, linewidth=0.16](4.925,-2.9775)(4.925,-3.0775)(4.825,-2.8775)(4.525,-2.8775)(4.325,-2.9775)(4.125,-3.1775)(4.025,-3.4775)(4.125,-3.7775)(4.125,-3.7775)
\psline[linecolor=black, linewidth=0.12](5.725,4.4225)(5.725,-4.7775)(5.725,-4.7775)
\psdots[linecolor=black, dotstyle=o, dotsize=0.4](5.725,4.4225)
\psdots[linecolor=black, dotstyle=o, dotsize=0.4](5.725,-4.7775)
\rput(11.125,0.2225){\psscalebox{3 3}{$Y$}}
\rput[b](5.725,5.1225){\psscalebox{3 3}{$u$}}
\rput[t](5.725,-5.2775){\psscalebox{3 3}{$v$}}
\end{pspicture}
}  
\psscalebox{.3.3} 
{
\begin{pspicture}(25,-4.9)(1.5,0)
\pscircle[linecolor=black, linewidth=0.22, dimen=outer](15.418964,-0.1){4.7}
\pscircle[linecolor=black, linewidth=0.04, dimen=outer](15.418964,-0.1){4.1}
\pscircle[linecolor=black, linewidth=0.04, dimen=outer](15.018964,3.4){0.6}
\pscircle[linecolor=black, linewidth=0.04, dimen=outer](15.698964,-3.58){0.6}
\pscircle[linecolor=black, linewidth=0.04, dimen=outer](12.888965,2.23){0.65}
\pscircle[linecolor=black, linewidth=0.04, dimen=outer](17.778965,2.44){0.6}
\pscircle[linecolor=black, linewidth=0.04, dimen=outer](12.918964,-2.5){0.6}
\psline[linecolor=white, linewidth=0.1](11.818964,1.6)(11.818964,1.7)(11.818964,1.3)(11.918964,1.4)(11.918964,1.4)
\pscircle[linecolor=white, linewidth=0.1, dimen=outer](13.658964,2.36){0.1}
\pscircle[linecolor=white, linewidth=0.1, dimen=outer](13.738964,3.48){0.1}
\pscircle[linecolor=white, linewidth=0.1, dimen=outer](12.878964,1.74){0.1}
\psline[linecolor=white, linewidth=0.1](11.818964,1.4)(11.918964,1.2)
\psline[linecolor=white, linewidth=0.1](12.818964,-1.8)(12.618964,-1.9)
\psline[linecolor=white, linewidth=0.1](13.618964,-2.4)(13.318964,-2.7)(13.318964,-2.7)
\psline[linecolor=white, linewidth=0.1](13.518964,-3.6)(13.618964,-3.7)
\psline[linecolor=white, linewidth=0.1](16.618965,3.7)(16.918964,3.6)(16.918964,3.6)
\psline[linecolor=white, linewidth=0.1](17.218964,3.7)(17.418964,3.4)(17.418964,3.4)
\psline[linecolor=white, linewidth=0.1](17.918964,2.2)(17.718964,2.0)(17.718964,1.9)
\psline[linecolor=white, linewidth=0.1](17.318964,-2.4)(17.618965,-2.7)
\psline[linecolor=white, linewidth=0.1](18.818964,2.1)(18.818964,1.9)
\psline[linecolor=white, linewidth=0.1](17.118965,2.8)(16.918964,2.5)
\psline[linecolor=white, linewidth=0.2](13.518964,-2.5)(13.718965,-2.6)
\psline[linecolor=white, linewidth=0.1](19.118965,-3.2)(18.918964,-3.5)(18.918964,-3.5)
\psline[linecolor=white, linewidth=0.1](19.118965,-1.6)(19.018965,-1.9)
\psline[linecolor=white, linewidth=0.1](13.518964,2.6)(13.518964,2.4)
\psline[linecolor=black, linewidth=0.04, arrowsize=0.04cm 4.0,arrowlength=3.0,arrowinset=0.0]{<-}(18.818964,2.2)(19.018965,1.8)
\psline[linecolor=black, linewidth=0.04, arrowsize=0.04cm 4.0,arrowlength=3.0,arrowinset=0.0]{<-}(17.718964,-3.5)(17.418964,-3.6)
\psline[linecolor=black, linewidth=0.04, arrowsize=0.04cm 4.0,arrowlength=3.0,arrowinset=0.0]{<-}(11.718965,-1.8)(11.618964,-1.5)
\psline[linecolor=black, linewidth=0.04, arrowsize=0.04cm 4.0,arrowlength=3.0,arrowinset=0.0]{<-}(17.718964,-3.5)(17.518965,-3.6)
\pscircle[linecolor=black, linewidth=0.04, dimen=outer](15.758965,3.36){0.6}
\pscircle[linecolor=black, linewidth=0.04, dimen=outer](15.038964,-3.62){0.6}
\psline[linecolor=black, linewidth=0.04](14.418964,3.4)(14.418964,-3.5)(14.418964,-3.5)
\psline[linecolor=black, linewidth=0.04](16.318964,3.3)(16.318964,-3.6)(16.318964,-3.6)
\psline[linecolor=black, linewidth=0.04](13.518964,2.2)(13.518964,-2.6)
\pscircle[linecolor=black, linewidth=0.04, dimen=outer](17.818964,-2.6){0.6}
\psline[linecolor=black, linewidth=0.04](17.258965,2.26)(17.258965,-2.54)
\psline[linecolor=white, linewidth=0.08](0.018964233,-1.3)(1.3189642,-2.0)
\psline[linecolor=white, linewidth=0.08](12.318964,2.7)(12.818964,3.0)(12.818964,3.0)
\psline[linecolor=white, linewidth=0.08](14.918964,4.1)(14.318964,3.7)
\psline[linecolor=white, linewidth=0.08](12.118964,-2.7)(13.018964,-3.4)
\psline[linecolor=white, linewidth=0.08](14.218965,-3.8)(14.718965,-4.2)
\psline[linecolor=white, linewidth=0.08](15.018964,-4.1)(15.718965,-4.1)
\psline[linecolor=white, linewidth=0.08](16.218964,-4.0)(15.818964,-4.3)
\psline[linecolor=white, linewidth=0.08](17.818964,-3.3)(18.318964,-3.1)(18.318964,-3.1)
\psline[linecolor=white, linewidth=0.08](17.718964,3.2)(18.318964,2.9)(18.318964,2.9)
\psline[linecolor=white, linewidth=0.08](12.418964,2.6)(12.218965,2.3)
\psline[linecolor=white, linewidth=0.08](12.318964,-2.7)(12.618964,-2.9)
\psline[linecolor=white, linewidth=0.08](13.418964,2.2)(13.418964,1.7)
\psline[linecolor=white, linewidth=0.08](13.418964,-1.9)(13.418964,-2.5)(13.418964,-2.5)
\psline[linecolor=white, linewidth=0.08](15.618964,4.1)(16.318964,3.8)(16.318964,3.8)
\psline[linecolor=white, linewidth=0.08](17.318964,2.4)(17.318964,2.0)
\psline[linecolor=white, linewidth=0.08](17.318964,-2.0)(17.318964,-2.6)
\psline[linecolor=white, linewidth=0.08](18.118965,2.8)(18.318964,2.7)
\psline[linecolor=white, linewidth=0.08](18.518965,-2.4)(18.418964,-2.2)
\psline[linecolor=white, linewidth=0.08](18.418964,-2.7)(18.318964,-2.9)(18.318964,-2.9)
\psline[linecolor=white, linewidth=0.08](15.118964,3.9)(15.618964,3.9)
\psline[linecolor=white, linewidth=0.08](14.518964,3.6)(14.618964,3.1)
\psline[linecolor=white, linewidth=0.08](14.518964,3.5)(14.518964,3.0)
\psline[linecolor=white, linewidth=0.08](16.218964,3.3)(16.218964,2.8)
\psline[linecolor=white, linewidth=0.08](14.518964,-3.1)(14.518964,-3.6)
\psline[linecolor=white, linewidth=0.08](14.518964,-3.4)(14.418964,-3.4)
\psline[linecolor=white, linewidth=0.08](16.218964,-3.0)(16.118965,-3.5)(16.118965,-3.5)(16.118965,-3.5)(16.218964,-3.3)(16.218964,-3.5)(16.018965,-3.0)
\psline[linecolor=white, linewidth=0.144](12.218965,2.3)(12.318964,2.0)(12.418964,1.8)(12.618964,1.7)(12.918964,1.6)(13.418964,1.8)(13.418964,1.8)
\psline[linecolor=white, linewidth=0.144](12.418964,-2.6)(12.318964,-2.6)(12.418964,-2.2)(12.618964,-2.0)(12.918964,-1.9)(13.018964,-1.9)(13.418964,-2.2)(13.418964,-2.4)
\psline[linecolor=white, linewidth=0.144](13.418964,2.4)(13.418964,1.8)
\psline[linecolor=white, linewidth=0.144](12.318964,-2.6)(12.318964,-2.6)
\psline[linecolor=white, linewidth=0.144](12.818964,3.1)(12.718965,3.0)(13.118964,3.3)(13.618964,3.5)(13.518964,3.6)(13.918964,3.7)(14.218965,3.8)(14.618964,3.9)(14.518964,3.8)
\psline[linecolor=white, linewidth=0.144](14.518964,3.4)(14.518964,3.1)
\psline[linecolor=white, linewidth=0.144](14.518964,-3.1)(14.518964,-3.6)
\psline[linecolor=white, linewidth=0.144](14.618964,-3.3)(14.718965,-3.0)(14.818964,-3.1)(15.018964,-3.0)(15.318964,-3.1)(15.518964,-3.3)(15.618964,-3.5)(15.618964,-3.7)(15.418964,-4.0)(15.518964,-4.0)(15.218965,-3.9)(15.118964,-3.6)(15.218965,-3.2)(15.618964,-2.9)(15.618964,-3.1)(15.818964,-3.0)(16.118965,-3.2)(16.218964,-3.2)(16.218964,-3.6)
\psline[linecolor=white, linewidth=0.144](14.618964,3.0)(14.518964,3.1)(14.818964,2.8)(15.118964,2.9)(15.118964,2.8)(15.418964,3.0)(15.618964,3.3)(15.618964,3.7)(15.318964,3.9)(14.918964,3.8)
\psline[linecolor=white, linewidth=0.144](15.418964,3.8)(15.218965,3.7)(15.218965,3.2)(15.318964,3.0)(15.618964,2.8)(16.018965,2.8)(16.218964,3.0)(16.218964,3.4)(16.218964,3.4)
\psline[linecolor=white, linewidth=0.144](15.518964,3.6)(15.518964,3.7)(15.618964,3.7)
\psline[linecolor=white, linewidth=0.144](15.218965,-4.0)(15.418964,-4.0)
\psline[linecolor=white, linewidth=0.144](17.318964,2.1)(17.618965,1.9)(18.018965,1.9)(18.318964,2.1)(18.318964,2.4)(18.318964,2.6)(18.318964,2.7)(18.318964,2.7)
\psline[linecolor=white, linewidth=0.144](13.018964,-3.3)(13.218965,-3.7)(13.218965,-3.5)(13.618964,-3.7)(13.818964,-3.9)(14.518964,-4.1)(14.518964,-4.1)
\psline[linecolor=white, linewidth=0.144](16.118965,-4.1)(16.218964,-4.1)(17.118965,-3.8)(17.318964,-3.8)(18.018965,-3.3)(17.118965,-3.8)(17.118965,-3.6)(17.418964,-3.5)(17.218964,-3.7)(17.318964,-3.6)
\psline[linecolor=white, linewidth=0.144](17.418964,-2.2)(17.518965,-2.1)(17.818964,-2.0)(18.118965,-2.1)(18.318964,-2.2)(18.218964,-2.3)(18.418964,-2.4)(18.418964,-2.6)(18.418964,-2.6)
\psline[linecolor=white, linewidth=0.144](17.418964,-2.2)(17.318964,-2.2)
\psline[linecolor=white, linewidth=0.144](16.018965,4.0)(16.618965,3.8)(17.118965,3.6)(18.018965,3.1)(18.018965,3.1)
\psline[linecolor=black, linewidth=0.12](13.918964,4.1)(13.918964,-4.5)(13.918964,-4.5)
\psline[linecolor=black, linewidth=0.12](16.818964,4.3)(16.818964,-4.6)
\psdots[linecolor=black, dotstyle=o, dotsize=0.4](13.918964,-4.4)
\psdots[linecolor=black, dotstyle=o, dotsize=0.4](16.818964,-4.5)
\psdots[linecolor=black, dotstyle=o, dotsize=0.4](16.818964,4.3)
\psdots[linecolor=black, dotstyle=o, dotsize=0.4](13.918964,4.2)
\psline[linecolor=black, linewidth=0.04, arrowsize=0.04cm 4.0,arrowlength=3.0,arrowinset=0.0]{<-}(16.318964,0.7)(16.318964,0.0)
\psline[linecolor=black, linewidth=0.04, arrowsize=0.04cm 4.0,arrowlength=3.0,arrowinset=0.0]{<-}(14.418964,0.1)(14.418964,0.6)
\psline[linecolor=black, linewidth=0.04, arrowsize=0.04cm 4.0,arrowlength=3.0,arrowinset=0.0]{<-}(13.518964,0.7)(13.518964,0.5)
\psline[linecolor=black, linewidth=0.04, arrowsize=0.04cm 4.0,arrowlength=3.0,arrowinset=0.0]{<-}(17.218964,0.1)(17.218964,0.5)
\rput(9.818964,0.3){\psscalebox{3 3}{$X$}}
\rput(20.918964,0.2){\psscalebox{3 3}{$Z$}}
\rput[b](15.218965,5.4){\psscalebox{3 3}{$Y_1$}}
\rput[t](15.318964,-5.4){\psscalebox{3 3}{$Y_2$}}
\rput[r](13.818964,5.1){\psscalebox{3 3}{$u'$}}
\rput[r](13.918964,-5.3){\psscalebox{3 3}{$u''$}}
\rput[l](16.818964,5.1){\psscalebox{3 3}{$v'$}}
\rput[l](16.918964,-5.3){\psscalebox{3 3}{$v''$}}
\end{pspicture}
}
\psscalebox{.3.3} 
{
\begin{pspicture}(20,-4.5)(3,0)
\pscircle[linecolor=black, linewidth=0.22, dimen=outer](6.24,0.2025){4.7}
\psline[linecolor=black, linewidth=0.12](3.04,3.3025)(9.84,-3.0975)(9.84,-3.0975)
\psline[linecolor=black, linewidth=0.12](2.94,-2.9975)(9.34,3.7025)(9.34,3.7025)
\psdots[linecolor=black, dotstyle=o, dotsize=0.4](9.24,3.7025)
\psdots[linecolor=black, dotstyle=o, dotsize=0.4](2.94,-2.9975)
\psdots[linecolor=black, dotstyle=o, dotsize=0.4](2.94,3.4025)
\psdots[linecolor=black, dotstyle=o, dotsize=0.4](9.64,-2.8975)
\rput[b](6.14,5.5025){\psscalebox{3 3}{$X$}}
\rput(6.14,-5.4975){\psscalebox{3 3}{$Z$}}
\rput[l](0.14,0.5025){\psscalebox{3 3}{$U$}}
\rput[r](12.14,0.5025){\psscalebox{3 3}{$Y$}}
\pscircle[linecolor=black, linewidth=0.04, dimen=outer](6.24,0.2025){4.1}
\pscircle[linecolor=black, linewidth=0.04, dimen=outer](7.64,3.4025){0.6}
\pscircle[linecolor=black, linewidth=0.04, dimen=outer](3.08,-1.3375){0.6}
\pscircle[linecolor=black, linewidth=0.04, dimen=outer](9.32,-1.3775){0.6}
\pscircle[linecolor=black, linewidth=0.04, dimen=outer](4.61,3.2325){0.65}
\pscircle[linecolor=black, linewidth=0.04, dimen=outer](3.2,1.9425){0.6}
\pscircle[linecolor=black, linewidth=0.04, dimen=outer](9.11,2.2325){0.55}
\pscircle[linecolor=black, linewidth=0.04, dimen=outer](8.1,-2.7575){0.6}
\pscircle[linecolor=black, linewidth=0.04, dimen=outer](4.44,-2.7975){0.6}
\psline[linecolor=black, linewidth=0.04](4.14,2.8025)(8.84,-1.6975)(8.84,-1.6975)
\psline[linecolor=black, linewidth=0.04](8.14,3.1025)(3.34,-1.8975)
\psline[linecolor=black, linewidth=0.04](3.64,2.3025)(8.54,-2.3975)(8.54,-2.3975)
\psline[linecolor=black, linewidth=0.04](8.74,2.6025)(3.94,-2.4975)(3.94,-2.4975)
\psline[linecolor=white, linewidth=0.1](2.64,1.9025)(2.64,2.0025)(2.64,1.6025)(2.74,1.7025)(2.74,1.7025)
\pscircle[linecolor=white, linewidth=0.1, dimen=outer](4.48,2.6625){0.1}
\pscircle[linecolor=white, linewidth=0.1, dimen=outer](4.56,3.7825){0.1}
\pscircle[linecolor=white, linewidth=0.1, dimen=outer](3.7,2.0425){0.1}
\psline[linecolor=white, linewidth=0.1](2.74,1.6025)(2.74,1.6025)(2.94,1.4025)(3.14,1.4025)(3.44,1.3025)(3.34,1.4025)(3.54,1.5025)(3.74,1.8025)(3.84,2.0025)(3.74,1.7025)
\psline[linecolor=white, linewidth=0.1](3.74,1.7025)(3.44,1.4025)
\psline[linecolor=white, linewidth=0.1](2.64,1.7025)(2.74,1.5025)
\psline[linecolor=white, linewidth=0.1](2.54,-1.2975)(2.54,-1.0975)(2.74,-0.7975)(2.74,-0.8975)(2.84,-0.6975)(2.84,-0.7975)(2.54,-1.2975)(2.54,-1.2975)
\psline[linecolor=white, linewidth=0.1](3.64,-1.4975)(3.44,-1.5975)
\psline[linecolor=white, linewidth=0.1](4.44,-2.0975)(4.14,-2.3975)(4.14,-2.3975)
\psline[linecolor=white, linewidth=0.1](4.34,-3.2975)(4.44,-3.3975)
\psline[linecolor=white, linewidth=0.1](4.54,3.8025)(4.74,3.9025)(4.74,3.9025)
\psline[linecolor=white, linewidth=0.1](7.44,4.0025)(7.74,3.9025)(7.74,3.9025)
\psline[linecolor=white, linewidth=0.1](8.04,4.0025)(8.24,3.7025)(8.24,3.7025)
\psline[linecolor=white, linewidth=0.1](8.74,2.5025)(8.54,2.3025)(8.54,2.2025)
\psline[linecolor=white, linewidth=0.1](9.34,2.8025)(9.74,2.5025)(9.74,2.5025)
\psline[linecolor=white, linewidth=0.1](9.54,-1.9975)(10.04,-1.5975)
\psline[linecolor=white, linewidth=0.1](2.44,-1.5975)(2.74,-1.8975)
\psline[linecolor=white, linewidth=0.1](3.74,-2.9975)(3.94,-3.1975)
\psline[linecolor=white, linewidth=0.1](8.14,-2.0975)(8.44,-2.3975)
\psline[linecolor=white, linewidth=0.1](8.74,-2.8975)(8.54,-3.3975)
\psline[linecolor=white, linewidth=0.1](9.64,2.4025)(9.64,2.2025)
\psline[linecolor=white, linewidth=0.1](7.94,3.1025)(7.74,2.8025)
\psline[linecolor=white, linewidth=0.2](7.44,4.0025)(7.24,3.7025)(7.04,3.3025)(7.24,3.1025)(7.54,2.8025)(7.74,2.9025)
\psline[linecolor=white, linewidth=0.2](4.64,3.9025)(5.14,3.7025)(5.14,3.3025)(5.34,3.0025)(5.14,2.9025)(4.94,2.7025)(4.54,2.6025)
\psline[linecolor=white, linewidth=0.2](2.84,-0.7975)(3.44,-0.7975)(3.54,-1.0975)(3.64,-1.4975)
\psline[linecolor=white, linewidth=0.2](4.44,-2.1975)(4.44,-2.1975)(4.84,-2.3975)(5.04,-2.7975)(4.84,-3.1975)(4.74,-3.2975)(4.44,-3.2975)(4.64,-3.2975)
\psline[linecolor=white, linewidth=0.2](4.34,-2.1975)(4.54,-2.2975)
\psline[linecolor=white, linewidth=0.2](8.14,-2.0975)(7.74,-2.3975)(7.54,-2.5975)(7.54,-2.7975)(7.64,-3.1975)(8.34,-3.2975)
\psline[linecolor=white, linewidth=0.2](8.64,-1.2975)(8.94,-1.5975)
\psline[linecolor=white, linewidth=0.2](8.54,2.3025)(8.84,1.8025)(9.24,1.7025)(9.54,2.0025)(9.64,2.2025)(9.64,2.2025)(9.64,2.2025)
\psline[linecolor=white, linewidth=0.2](8.74,-1.2975)(8.94,-0.8975)(9.34,-0.7975)(9.64,-0.8975)(9.84,-1.0975)(9.84,-1.2975)(9.74,-1.4975)
\psline[linecolor=white, linewidth=0.2](9.74,-1.8975)(8.74,-2.9975)(9.14,-2.6975)
\psline[linecolor=white, linewidth=0.2](4.14,3.9025)(3.94,3.6025)
\psline[linecolor=white, linewidth=0.2](2.64,2.4025)(3.04,2.7025)
\psline[linecolor=white, linewidth=0.2](2.94,2.8025)(3.24,2.8025)(3.14,3.1025)(3.34,3.1025)(3.94,3.7025)(3.94,3.7025)
\psline[linecolor=white, linewidth=0.2](2.64,-1.8975)(3.74,-3.0975)(3.74,-3.0975)
\psline[linecolor=white, linewidth=0.2](8.34,3.8025)(9.04,3.1025)(9.44,2.8025)(9.44,2.8025)
\psline[linecolor=white, linewidth=0.2](8.24,3.9025)(8.34,3.7025)
\psline[linecolor=black, linewidth=0.12](3.18,-2.7375)(9.08,3.4654675)(9.043125,3.5025)
\psline[linecolor=black, linewidth=0.12](3.08,3.2625)(9.48,-2.7375)(9.48,-2.7375)
\psline[linecolor=white, linewidth=0.1](9.94,-2.8975)(9.74,-3.1975)(9.74,-3.1975)
\psline[linecolor=white, linewidth=0.1](9.94,-1.2975)(9.84,-1.5975)
\psline[linecolor=white, linewidth=0.1](4.34,2.9025)(4.34,2.7025)
\psline[linecolor=black, linewidth=0.04, arrowsize=0.04cm 4.0,arrowlength=3.0,arrowinset=0.0]{<-}(7.84,2.8025)(7.64,2.6025)
\psline[linecolor=black, linewidth=0.04, arrowsize=0.04cm 4.0,arrowlength=3.0,arrowinset=0.0]{<-}(8.84,-1.6975)(8.54,-1.3975)
\psline[linecolor=black, linewidth=0.04, arrowsize=0.04cm 4.0,arrowlength=3.0,arrowinset=0.0]{<-}(3.74,2.2025)(3.94,2.0025)
\psline[linecolor=black, linewidth=0.04, arrowsize=0.04cm 4.0,arrowlength=3.0,arrowinset=0.0]{<-}(4.14,-2.2975)(4.54,-1.8975)
\psline[linecolor=black, linewidth=0.04, arrowsize=0.04cm 4.0,arrowlength=3.0,arrowinset=0.0]{<-}(9.64,2.5025)(9.84,2.1025)
\psline[linecolor=black, linewidth=0.04, arrowsize=0.04cm 4.0,arrowlength=3.0,arrowinset=0.0]{<-}(8.54,-3.1975)(8.24,-3.2975)
\psline[linecolor=black, linewidth=0.04, arrowsize=0.04cm 4.0,arrowlength=3.0,arrowinset=0.0]{<-}(2.54,-1.4975)(2.44,-1.1975)
\psline[linecolor=black, linewidth=0.04, arrowsize=0.04cm 4.0,arrowlength=3.0,arrowinset=0.0]{<-}(4.34,3.8025)(5.04,4.1025)
\psline[linecolor=black, linewidth=0.04, arrowsize=0.04cm 4.0,arrowlength=3.0,arrowinset=0.0]{<-}(8.54,-3.1975)(8.34,-3.2975)
\psline[linecolor=white, linewidth=0.2](5.64,0.8025)(5.94,0.6025)
\psline[linecolor=white, linewidth=0.2](6.04,0.5025)(6.34,0.2025)(6.34,0.2025)
\psline[linecolor=white, linewidth=0.2](5.64,0.9025)(6.54,0.0025)
\psline[linecolor=white, linewidth=0.2](5.94,1.1025)(6.84,0.2025)
\psline[linecolor=white, linewidth=0.2](5.44,0.6025)(6.34,-0.2975)
\psline[linecolor=black, linewidth=0.12](6.74,1.0025)(5.14,-0.6975)(5.14,-0.6975)
\psline[linecolor=black, linewidth=0.04](6.24,1.1025)(5.44,0.3025)
\psline[linecolor=black, linewidth=0.04](6.84,0.6025)(6.04,-0.2975)
\rput(2.04,4.4025){\psscalebox{3 3}{$u'$}}
\rput(10.44,-3.6975){\psscalebox{3 3}{$u''$}}
\rput(2.04,-3.6975){\psscalebox{3 3}{$v''$}}
\rput(9.94,4.6025){\psscalebox{3 3}{$v'$}}
\end{pspicture}
}
\caption{\label{fig:zitra_do_Tabora_na_Vivat_tour}Three  pancakes.}
 \end{figure}
When its circumference is labeled by the ordered sets $X$ and $Y$ as
in the figure, then, by definition
\[
\varpi_K := \xi_{uv}\varsigma\big([XuYv]\big) \in
\nsP\bl[X],[Y];1\br.
\]
We  prove that $\varpi_K$ depends only
on the induced cyclically ordered sets $[X]$ and $[Y]$, not on the
particular orders of $X$ and $Y$.

Let, for instance, $X'$ be an ordered set such that $[X'] = [X]$ and $K'$ be
the pancake obtained from $K$ by replacing $X$ by $X'$.
We will show that
\begin{equation}
  \label{eq:Jaruska_porad_na_chlupe}
\varpi_K = \varpi_{K'},  
\end{equation}
where 
\[
\varpi_{K'} := \xi_{uv}\varsigma\big([X'uYv]\big) \in
\nsP\bl[X'],[Y];1\br = \nsP\bl[X],[Y];1\br.
\]

As we noticed in Remark~\ref{zase_mam_nutkani}, $[X'] = [X]$ if and
only if there are ordered sets $X_1$ and $X_2$ such that $X = X_1X_2$
and $X' = X_2X_1$. By the interchange (iv) of
Definition~\ref{opet_podleham},
\begin{equation}
\label{Jaruska_prijede_snad_v_nedeli}
\xi_{v'v''}\big\{\varsigma([u'v'X_1]) \ooo {u'}{u''}
\varsigma([X_2v''Yu''])\big\}
=
\xi_{u'u''}\big\{\varsigma([u'v'X_1]) \ooo {v'}{v''}
\varsigma([X_2v''Yu''])\big\}. 
\end{equation}
The corresponding term in the curly bracket in the left hand side equals 
\[
\varsigma\big([u'v'X_1]\big) \ooo {u'}{u''} \varsigma\big([X_2v''Yu'']\big)
= \varsigma\big([v'X_1X_2v''Y]\big)= \varsigma\big([X_1X_2v''Yv']\big) =
\varsigma\big([Xv''Yv']\big),
\]
while the term in the right hand side is
\[
\varsigma\big([u'v'X_1]\big) \ooo {v'}{v''} \varsigma\big([X_2v''Yu'']\big) =
\varsigma\big([X_1u'Yu''X_2]\big) = \varsigma\big([X_2X_1u'Yu'']\big) = \varsigma\big([X'u'Yu'']\big),
\]
thus~(\ref{Jaruska_prijede_snad_v_nedeli}) implies
\[
 \xi_{v'v''}\varsigma\big([Xv''Yv']\big) =\xi_{u'u''}\varsigma\big([X'u'Yu'']\big)
\]
which is~(\ref{eq:Jaruska_porad_na_chlupe}).  The independence of the
particular order of $Y$ can be proved similarly.
\end{example}

\begin{example}
\label{sec:modular-envelopes-3}
There are two types of  pancakes with $g=2$. The middle one in 
Figure~\ref{fig:zitra_do_Tabora_na_Vivat_tour} has
$b=3$ and 
\begin{subequations}
\begin{equation}
\label{eq:17}
\varpi_K
=\xi_{u'u''}\xi_{v'v''}\varsigma\big([Xu'Y_1v'Zv''Y_2u'']\big) \in
\nsP\bl[X],[Y_1Y_2],[Z];3\br.
\end{equation}
The second type
in the right hand side of
Figure~\ref{fig:zitra_do_Tabora_na_Vivat_tour} has $b=1$ and
\begin{equation}
\label{eq:18}
\varpi_K
=\xi_{u'u''}\xi_{v'v''}\varsigma\big(Xv' Yu''Zv''Uu'\big) \in
\nsP\bl[{\it UZYX}];2\br.
\end{equation}
\end{subequations}
We leave as an exercise on the axioms of \ns\ modular operads to prove
that the elements $\varpi_K$ in~(\ref{eq:17}) resp.\ in~(\ref{eq:18})
depend only on the cyclically ordered sets $[X]$, $[Y_1Y_2]$ and $[Z]$
resp.~$[{\it UZYX}]$.
\end{example}

We prove~(\ref{eq:1bis1}) by induction based on the following simple:

\begin{lemma}
\label{sec:modular-envelopes-2}
Let $K$ be a  pancake with $g > 0$. 
\begin{itemize}
\itemindent -1em
 \itemsep .3em 
\item [(i)]
If $b(K)=1$, then there exist a  pancake $K'$ obtained  by
removing one rib of $K$ such that  $b(K') = 2$.
\item [(ii)]
If $b(K)>1$, then there exist a  pancake $K'$ obtained from $K$ by
removing one rib that has $b(K') = b(K)-1$.
\end{itemize}
\end{lemma}

\begin{proof}
  The case $b(K)=1$ may happen only when $g(K)$ is even, by the 
geometricity~(\ref{eq:Jaruska}). Thus by removing
  an {\em arbitrary\/} rib we obtain a pancake $K'$ with $b(K') =
  2$. This proves (i).

To prove (ii) we analyze the
pancake~(\ref{Jaruska_se_uz_na_chalupe_prudi.}) representing $K$. 
The legs of $K$ are irrelevant for the proof so we
ignore them. Let us inspect how the ribs enter the circumference of
the pancake.
The
oriented edge cycles~(\ref{zasil_jsem_polstarky}) 
are the boundaries of the regions between the ribs:
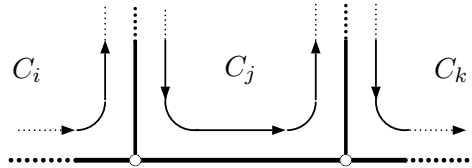
\begin{center}
\psscalebox{.4.4} 
{
\begin{pspicture}(0,-2.696111)(15.213337,2.696111)
\psline[linecolor=black, linewidth=0.16](2.106667,-2.494722)(13.106667,-2.494722)
\psline[linecolor=black, linewidth=0.16, linestyle=dotted, dotsep=0.10583334cm](13.006667,-2.494722)(15.206667,-2.494722)
\psline[linecolor=black, linewidth=0.16, linestyle=dotted, dotsep=0.10583334cm](2.206667,-2.494722)(-0.093332976,-2.494722)
\psline[linecolor=black, linewidth=0.12](4.106667,1.5052781)(4.106667,-2.494722)
\psline[linecolor=black, linewidth=0.12](11.106667,1.5052781)(11.106667,-2.494722)
\psline[linecolor=black, linewidth=0.12, linestyle=dotted, dotsep=0.10583334cm](4.106667,1.4052781)(4.106667,2.605278)
\psline[linecolor=black, linewidth=0.12, linestyle=dotted, dotsep=0.10583334cm](11.106667,1.5052781)(11.106667,2.7052782)
\psline[linecolor=black, linewidth=0.06, arrowsize=0.2cm 2.0,arrowlength=1.4,arrowinset=0.0]{<-}(5.106667,-0.49472192)(5.106667,1.5052781)
\psline[linecolor=black, linewidth=0.06, arrowsize=0.2cm 2.0,arrowlength=1.4,arrowinset=0.0]{<-}(10.106667,1.5052781)(10.106667,-0.49472192)
\psline[linecolor=black, linewidth=0.06, arrowsize=0.2cm 2.0,arrowlength=1.4,arrowinset=0.0]{<-}(9.106667,-1.4947219)(6.106667,-1.4947219)
\psline[linecolor=black, linewidth=0.06, arrowsize=0.2cm 2.0,arrowlength=1.4,arrowinset=0.0]{<-}(3.106667,1.5052781)(3.106667,-0.49472192)
\psline[linecolor=black, linewidth=0.06, arrowsize=0.2cm 2.0,arrowlength=1.4,arrowinset=0.0]{<-}(12.106667,-0.49472192)(12.106667,1.5052781)
\psline[linecolor=black, linewidth=0.06, linestyle=dotted, dotsep=0.10583334cm, arrowsize=0.2cm 2.0,arrowlength=1.4,arrowinset=0.0]{<-}(2.106667,-1.4947219)(0.10666702,-1.4947219)
\psline[linecolor=black, linewidth=0.06, linestyle=dotted, dotsep=0.10583334cm, arrowsize=0.2cm 2.0,arrowlength=1.4,arrowinset=0.0]{<-}(15.106667,-1.4947219)(13.106667,-1.4947219)
\psarc[linecolor=black, linewidth=0.06, dimen=outer](6.056667,-0.5447219){0.95}{163.9798}{279.5601}
\psarc[linecolor=black, linewidth=0.06, dimen=outer](13.056667,-0.5447219){0.95}{163.9798}{279.5601}
\psarc[linecolor=black, linewidth=0.06, dimen=outer](2.156667,-0.5447219){0.95}{270}0
\psarc[linecolor=black,linewidth=0.06,dimen=outer](9.156667,-0.5447219){0.95}{270}{0}
\psline[linecolor=black, linewidth=0.06, linestyle=dotted, dotsep=0.10583334cm](5.106667,1.5052781)(5.106667,2.505278)
\psline[linecolor=black, linewidth=0.06, linestyle=dotted, dotsep=0.10583334cm](10.106667,1.3052781)(10.106667,2.7052782)
\psline[linecolor=black, linewidth=0.06, linestyle=dotted, dotsep=0.10583334cm](3.106667,1.3052781)(3.106667,2.605278)
\psline[linecolor=black, linewidth=0.06, linestyle=dotted, dotsep=0.10583334cm](12.106667,1.5052781)(12.106667,2.7052782)
\rput(0.506667,0.5052781){\psscalebox{2.5 2.5}{$C_i$}}
\rput(7.606667,0.5052781){\psscalebox{2.5 2.5}{$C_j$}}
\rput(14.606667,0.5052781){\psscalebox{2.5 2.5}{$C_k$}}
\psdots[linecolor=black, dotstyle=o, dotsize=0.4](4.106667,-2.494722)
\psdots[linecolor=black, dotstyle=o, dotsize=0.4](11.106667,-2.494722)
\end{pspicture}
}
\end{center}
In this picture, the horizontal line represents a part of the
circumference of the pancake and the vertical lines the ribs.

It is  simple to see that removing a rib
adjacent to two different regions decreases $b(K)$ by
one. Since $b(K) > 1$ by assumption, there are at least two different
edge cycles, therefore such a rib exists. This finishes the proof. 
\end{proof}

Let us finally start the actual inductive proof of~(\ref{eq:1bis1}).
The cases when $g(K) \leq 2$ are analyzed in
Examples~\ref{sec:modular-envelopes-5}--\ref{sec:modular-envelopes-3}. Fix
$g \geq 3$, assume that we have proved~(\ref{eq:1bis1})
for all $K$'s with $g(K) < g$ and prove it for $K$ with $g(K) = g$.  
As in Lemma~\ref{sec:modular-envelopes-2} we distinguish two cases.

\vskip .5em

\noindent {\bf Case 1: $b(K) = 1$.} As there are no pancakes
with $b(K) = 1$ and $g(K) = 4$, in this case in fact $g(K) \geq 4$.
By Lemma~\ref{sec:modular-envelopes-2}.(i) and the
inductive assumption,
\begin{equation}
\label{eq:19}
\varpi_K = \xi_{u'u''}
\varpi\big([X_1u'],[u''X_2];g-1\big),
\end{equation}
where $X_1$, $X_2$ are ordered sets such that $C := \Leg(K) =
[X_1X_2]$. We must show
that the right hand side of~(\ref{eq:19}) 
does not depend on the particular choices of $X_1$ and $X_2$. The
choices are 
represented by a rib of a circle with the circumference decorated
by $C$ as in Figure~\ref{fig:cut}--left.\begin{figure}
\centering
\psscalebox{.25.25}
{ 
\begin{pspicture}(0,-4.2225)(34.66,6.2225)
\rput(-4,0){
\pscircle[linecolor=black, linewidth=0.22, dimen=outer](4.98,0.6025){4.7}
\psline[linecolor=black, linewidth=0.12](1.02,4.6625)(9.52,-3.4375)(9.52,-3.4375)
\psdots[linecolor=black, dotstyle=o, dotsize=0.7](1.82,3.8625)
\psdots[linecolor=black, dotstyle=o, dotsize=0.7](8.42,-2.4375)
\rput[b](1.78,5.0025){\psscalebox{4 4}{{$u'$}}}
\rput[t](8.38,-3.3975){\psscalebox{4 4}{{$u''$}}}
\rput(0.18,-2.7975){\psscalebox{4 4}{{$X_1$}}}
\rput(9.48,4.7025){\psscalebox{4 4}{{$X_2$}}}
}

\psline[linecolor=black, linewidth=0.12](13.58,-3.5975)(19.48,5.8025)(19.410587,5.8025)
\pscircle[linecolor=black, linewidth=0.22, dimen=outer](17.62,0.5625){4.7}
\psline[linecolor=black, linewidth=0.12](13.66,5.1225)(23.460613,-1.5765196)(13.691611,5.1021967)
\psdots[linecolor=black, dotstyle=o, dotsize=0.7](22.06,-0.5775)
\psdots[linecolor=black, dotstyle=o, dotsize=0.7](14.3,-2.5175)
\psdots[linecolor=black, dotstyle=o, dotsize=0.7](18.94,4.9425)
\psdots[linecolor=black, dotstyle=o, dotsize=0.7](14.88,4.3025)
\rput[bl](15.08,-2.5){\psscalebox{4 4}{{$v''$}}}
\rput[b](18.48,5.8025){\psscalebox{4 4}{{$v'$}}}
\rput[rt](21.5,-.8){\psscalebox{4 4}{{$u''$}}}
\rput[t](15,3.5){\psscalebox{4 4}{{$u'$}}}
\rput(16.08,6.1025){\psscalebox{4 4}{{$X$}}}
\rput(12,0.9025){\psscalebox{4 4}{{$Y$}}}
\rput(17.88,-5.1){\psscalebox{4 4}{{$Z$}}}
\rput(22.4,3.8025){\psscalebox{4 4}{{$U$}}}

\rput(4,0){
\pscircle[linecolor=black, linewidth=0.22, dimen=outer](29.96,0.5225){4.7}
\psline[linecolor=black, linewidth=0.12](27.783419,5.3822923)(27.800613,-4.6165195)(27.78,-4.596912)
\psline[linecolor=black, linewidth=0.12](30.24,5.8425)(32.140614,-4.4565196)(30.24613,5.8112864)
\psdots[linecolor=black, dotstyle=o, dotsize=0.7](32.02,-3.5375)
\psdots[linecolor=black, dotstyle=o, dotsize=0.7](27.76,4.6225)
\psdots[linecolor=black, dotstyle=o, dotsize=0.7](27.8,-3.5175)
\psdots[linecolor=black, dotstyle=o, dotsize=0.7](30.36,5.1225)
\rput(24.28,0.8025){\psscalebox{4 4}{{$Y$}}}
\rput(29,6.1025){\psscalebox{4 4}{{$X$}}}
\rput(34.38,4.4025){\psscalebox{4 4}{{$U$}}}
\rput(29.98,-5.1975){\psscalebox{4 4}{{$Z$}}}
\rput[b](26.7,5.0025){\psscalebox{4 4}{{$v'$}}}
\rput[t](26.5,-3.8975){\psscalebox{4 4}{{$v''$}}}
\rput[b](31.2,5.8025){\psscalebox{4 4}{{$u'$}}}
\rput[t](33.5,-3.75){\psscalebox{4 4}{{$u''$}}}
}
\end{pspicture}
}
\caption{\label{fig:cut}
One rib (left), two crossing ribs (middle) and two parallel
ribs (left).}
\end{figure}
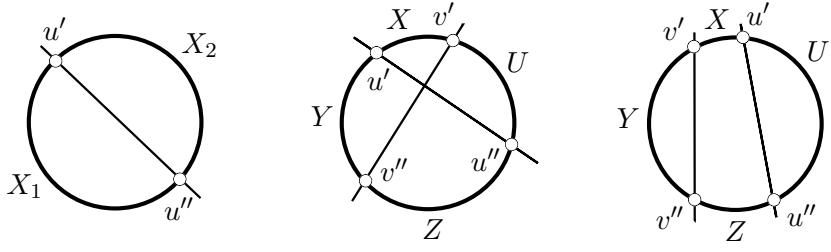

Assume we have two different ribs, $C =[X'_1X'_2]$ and $C
=[X''_1X''_2]$.  They may either be crossing as in the middle of
Figure~\ref{fig:cut}, or parallel as in the rightmost picture of
Figure~\ref{fig:cut}.
The crossing case is parametrized by ordered sets $X,Y,Z,U$ such that
\[
X'_1 = {\it ZY},\ X'_2 = {\it XU},\ X''_1 = {\it YX} \ \hbox { and } \
    X''_2 = {\it UZ},
\] 
see Figure~\ref{fig:cut} again.
The element 
\[
\varpi\big([v''Yu''Zv'Uu'X];g-2\big) \in \nsP\bl[v''Yu''Zv'Uu'X];g-2\br
\]
is then the `equalizer' of the ribs,
which is expressed by the commutative square
\[
\xymatrix@C = -5em@R = 2em{  
& \varpi\big([v''Yu''Zv'Uu'X];g-2\big) \ar@{|->}[dl]_{\xi_{v'v''}}
\ar@{|->}[dr]^{\xi_{u'u''}}   & 
\\
\varpi\big([X'_1u''],[u'X'_2] ; g-1\big)\ar@{|->}[dr]_{\xi_{u'u''}} & & \
\varpi\big([X''_1v''],[v'X''_2] ; g-1\big).
\ar@{|->}[dl]^{\xi_{v'v''}}
\\
&\varpi_K&
}
\]
Therefore 
\[
 \xi_{u'u''} \varpi\big([X'_1u'],[u''X'_2];g-1\big) = \xi_{u'u''}
\varpi\big([X''_1u'],[u''X''_2];g-1\big) 
\]
as we needed to show.
Notice that we need to assume $g \geq 2$ in order the equalizer to exist.

The non-crossing case is parametrized by ordered sets $X,Y,Z,U$ such
that
\[
X'_1 = {\it ZYX},\ X'_2 = U,\ X''_1 = {\it XUZ} \ \hbox { and } \ X''_2 = Y.
\] 
The equalizer of these two choices is 
$\varpi\big([Xu'Zv'],[v''Y],[u''U];g-2\big)$ as the reader easily
verifies. This finishes the discussion of the $b(K) = 1$ case.

\vskip .5em

\noindent 
{\bf Case 2: $b(K) > 1$.}
Now $K$ is of type $(C_1,\ldots,C_b;g)$ with $b \geq 2$. 
Lemma~\ref{sec:modular-envelopes-2}.(ii) translates to the formula 
\begin{equation}
\label{eq:16bis}
\varpi_K =
\xi_{u',u''}\varpi\big(C_1,\ldots,\widehat{C_i},
\ldots,\widehat{C_j},\ldots,C_b,[u'X_iu''X_j];g-1\big),
\end{equation}
where $\widehat{\hphantom{-}}$ indicates the omission and $X_i$, $X_j$
are ordered sets such that $[X_i]= C_i$ and $[X_j] = C_j$.
As before we must  prove that the right hand side of~(\ref{eq:16bis}) 
does not depend on the particular choices of $i$, $j$ and ordered sets
$X_i$, $X_j$. So suppose that we have two different choices
\[
[X'_i] = C_{i'}, \ [X'_j] = C_{j'}, \ \hbox{ resp.\ } \ 
[X''_i] = C_{i''},\ 
[X''_j] = C_{j''}
\] 
for some ordered sets $X'_i,X'_j, X''_i, X''_j$
and $1 \leq i' < j' \leq b$,  $1 \leq i'' < j'' \leq b$. We
distinguish three cases.

\vskip .3em

\noindent 
{\it Case 2(i): $\{i',j'\} = \{i'',j''\}$.}
We may clearly assume that $i' \! = \!i''= \!1$, $j'\!=\!j''\!=\!2$. 
Since $C_k$'s with $k > 2$ do not affect calculations we will not
explicitly mention them.
We therefore have 
\[
C_1 = [X'_1] =  [X''_1], \ C_2 = [X'_2] =  [X''_2],
\]
which, as observed in Remark~\ref{zase_mam_nutkani}, happens if and
only if there are ordered sets $X,Y,Z,U$ such that
\[
X'_1 = XY,\ X'_2 = ZU,\ X''_1 = YX \  \hbox { and } \ X''_2 = UZ.
\]
One easily verifies that then
$\varpi\big([Yv''Zu'],[u''Uv'X];g-2\big)$ is a equalizer of these choices.

\vskip .3em

\noindent 
{\it Case 2(ii): the cardinality of $\{i',j',i'',j''\}$ is $3$.}
We may assume that $i'\!= \!1$, $j'\! =\! i''\! = \!2$, $j''\! = \!3$,
and neglect $C_k$'s with $k >3$. So we have two presentations
\begin{equation}
  \label{eq:16}
\varpi_K = \varpi\big([u'X_1u''X'_2],C_3; g-1\big) \ \hbox { and } \
\varpi_K = \varpi\big(C_1,[v'X''_2v''X_3]; g-1\big)
\end{equation}
in which $C_1 = [X_1],\ C_2 = [X'_2] = [X''_2]$ and $C_3 = [X_3]$. By
Remark~\ref{zase_mam_nutkani}, there are ordered sets $Y$, $Z$ such
that $X'_2 = YZ$ and $X''_2 = ZY$. One easily sees that
$\varpi\big(u'X_1u''Yv'X_3v''Z; g-2\big)$ is a equalizer for the two
presentations in~(\ref{eq:16}).
The last case is

\vskip .3em

\noindent 
{\it Case 2(iii): the cardinality of $\{i',j',i'',j''\}$ is $4$.}
This case is simple  so we leave its analysis to the reader.
This finishes the proof of Theorem~\ref{sec:modular-envelopes-1}.

\begin{remark}
The computations in this section, namely in
Example~\ref{sec:modular-envelopes-4}, can be used to lift the
commutativity assumption in \cite[Section~4.1]{Kaufmann-FeynmanII},
cf.\ namely Remark~4.1 of \cite{Kaufmann-FeynmanII}, proving that
TFT's as functors from the category of cobordisms to
vector spaces are equivalent to  Frobenius algebras.
\end{remark}

\section{Surface models of $\Mod(\Ass)$.}
\label{lazarev}

In this section we recall the approach which J.~Chuang and A.~Lazarev
used to prove their~\cite[Theorem~3.7]{chuang-lazarev:dual-fey} that
describes the modular envelope $\Mod(\Ass)$ via the set of isomorphism
classes of oriented surfaces with teethed holes.  We also present a
parallel approach due to R.M.~Kaufmann,
M.~Livernet and 
R.C.~Penner~\cite{kaufmann-livernet-penner,kaufmann-penner:NP06}.  
Our setup nicely conveys their ideas.  The above mentioned
authors of course worked in the category of ordinary operads, but what
they did in fact took place within \nsm\ operads.  In the second half
of this section we briefly mention a non-oriented modification due to
C.~Braun~\cite{braun}.  Throughout this section, the basic monoidal
category will be the cartesian category $\Set$ of sets.

Let $\underline{*}$ be the geometric \nsm\  module of
Example \ref{zase_mne_boli_v_krku_tentokrat_jsem_to_chytl_od_Andulky}.
Its generators $\underline *\(C\)$ will this time be visualized as
oriented cog wheels
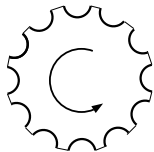
\begin{equation}
\raisebox{-2.4em}{}
\label{Pozitri_se_sem_nastehuje_Andulka.}
\psscalebox{.3.3}
{

}
\end{equation}
whose cogs are indexed by the cyclically ordered set $C$. The
elements of the free \nsm\ operad $\Free(\underline*)$ are then
obtained by gluing these wheels together along the tips of their cogs
so that the orientation is preserved, see
Figure~\ref{fig:orandnonor}--left.
\begin{figure}[t]
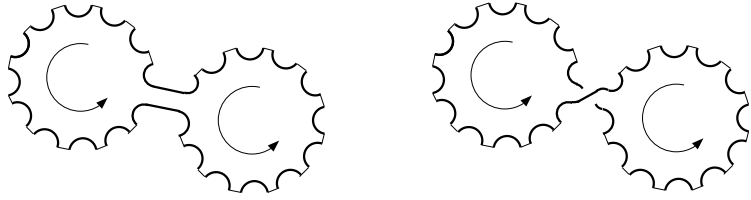

\psscalebox{.3 .3} %
{
%
}
  \caption{\label{fig:orandnonor}Oriented (left) and un-oriented (right) glueing of cog wheels.}
 
\end{figure}
It is an  exercise in combinatorial geometry that
 $\Free(\underline*)(C_1,\ldots,C_b;g)$ consists of all decompositions of
an oriented surface of the genus 
\begin{equation}
  \label{eq:jarkoprijed}
G = \textstyle \frac12 (g-b+1)
\end{equation}
with $b$ teethed boundaries whose teeth are labelled by the cyclically ordered
sets $\Rada C1b$ shown in Figure~\ref{zubata_plocha}. 

\begin{figure}
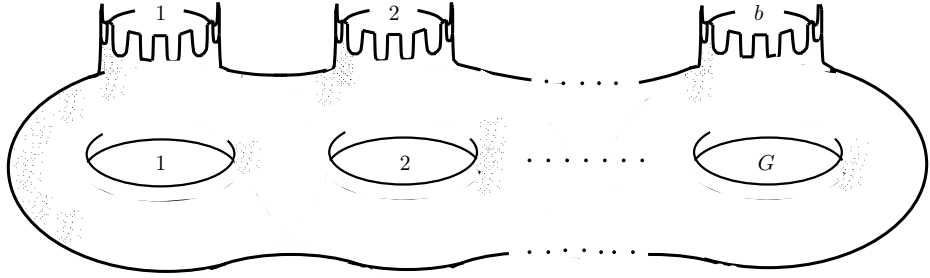

\begin{center}
\psscalebox{.25.25}
{
%
}
\end{center}
\caption{\label{zubata_plocha}The oriented surface with $b$ teethed boundaries and
  genus $G$.}
\end{figure}

By~(\ref{uplne_na_mne_masti}), the modular envelope
$\nsMod(\underline*_C)$ is the quotient of $\Free(\underline*) =
\Free\big(\nsF(\underline*_C)\big)$ by the relations that forget how a
concrete surface was build from the cog wheels. As in the proof
of~\cite[Theorem~3.7]{chuang-lazarev:dual-fey} we therefore identify,
referring to the results of~\cite{igusa}, $\nsMod(\underline*_C)$ with
the set of isomorphism classes of surfaces as in
Figure~\ref{zubata_plocha}. Since there is
only one isomorphism class for a given geometric
$(C_1,\ldots,C_b;g) \in \hbox{$\MultCyc \!\times\! \bbN$}$, we get~(\ref{Jarka_stale_bez_prace-b}).

\begin{remark}
\label{je_mi_trochu_lip}
Notice that, given $(C_1,\ldots,C_b;g) \in \MultCyc \!\times\! \bbN$,
there exists a surface as in Figure~\ref{zubata_plocha} if and only if
$(C_1,\ldots,C_b;g)$ is geometric. This shall explain our terminology.
\end{remark}

Other incarnations of the surface model of the operad
$\nsMod(\underline*_C)$ appeared also e.g.\ in
~\cite{kaufmann-livernet-penner,kaufmann-penner:NP06}. Let us outline its basic features.
Our exposition will be merely sketched and also the notation
will serve only for the purposes of this section.
Denote by $\Wind$  the modular operad
of isomorphism classes of surfaces as in Figure~\ref{zubata_plocha},
with the operad structure given by  glueing the tips of the
teeth. This operad is a version of the operad  
of isomorphism classes of windowed surfaces 
introduced in~\cite[Section~1]{kaufmann-penner:NP06}, with the r\^ole
of teeth played by marked points on the boundary. 

The operad $\Arc$ of isotopy classes of arc families in windowed
surfaces, see~\cite[Section~1]{kaufmann-livernet-penner} or
\cite[Section~2]{kaufmann-penner:NP06}, contains the suboperad
$\Triang$ of arc families whose complementary regions triangulate the
underlying surface.  One has the commutative diagram
\[
\xymatrix{\nsMod(\underline*_C) \ar[rr]^\omega_\cong & &  {\Wind}   &
\\
& {\Triang}\ \ar@{->>}[lu]_\beta \ar@{->>}[ru]^\alpha \ar@{^{(}->}[rr]&& 
{\Arc}\ar@{->>}[lu]_\gamma
}
\]
in which $\gamma : \Arc \to \Wind$ associates to an arc family its
underlying surface, and $\alpha : \Triang \to \Wind$ is the
restriction of $\gamma$.  Both maps are surjective because every
surface has a triangulation. The existence of the third map $\beta:
\Triang \to \nsMod(\underline*_C)$ follows from the freeness $\Triang$
while $\omega$ exists because $\Wind$ contains $\underline*_C$ as a
cyclic suboperad. The morphism $\beta$ is surjective as well, since
$\Triang$ contains the generators of the terminal \ns\ cyclic operad
$\underline*_C$.

To establish that $\omega$ is an isomorphism, one needs to know that 
the kernel of $\alpha$ contains the kernel of $\beta$. In geometric terms
this means that all triangulations of a given surface differ by a
sequence of `elementary moves' corresponding to the axioms of modular
operads.

Despite its conceptual clarity, the above approaches relied on a
rather deep result of~\cite{igusa}  that the classifying space
of the category of ribbon graphs of genus $G$ with $b$ boundary
components is homeomorphic to the moduli space
of Riemann surfaces of the same genus and the same number of 
boundary components. We therefore still believe that a direct combinatorial
description of $\Mod(\Ass)$ given in \cite{doubek:modass} or here has some merit.

\begin{remark}
  There are three equivalent pictures of the surface model for
  $\nsMod(\underline*_C)$: teethed surfaces in
  Figure~\ref{zubata_plocha} and the windowed surfaces with marked points on
  the boundary, resp.\ appropriate subdivisions of these surfaces
  given by arc families. The third, dual picture, closer to the
  approach of the present article, uses graphs whose vertices
  correspond to the regions of the subdivision and the edges
  to the their common boundaries. This remark should 
  relate e.g.~\cite{barannikov,kaufmann:dim_vs_genus,kaufmann-penner:NP06}
  where these objects were used in the operadic context,
  to the present article.
\end{remark}

A non-oriented variant of the above calculations starts with
the cyclic operad  $*_D$ whose component $*_D\(S\)$ consists of
cog wheels  whose cogs are indexed by the finite set $S$ and have their tips
decorated by  arrows, as 
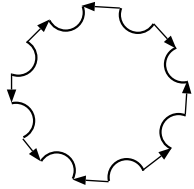
\begin{center}
\psscalebox{.3.3}
{
\begin{pspicture}(0,-3.926391)(7.8518815,3.926391)
\psarc[linecolor=black, linewidth=0.1, dimen=outer](2.43,3.6035507){0.8}{206}{20}
\psarc[linecolor=black, linewidth=.1, dimen=outer](0.23,-1.1964494){0.8}{294}{104}
\psarc[linecolor=black, linewidth=.1, dimen=outer](5.63,3.5035505){0.8}{158.33264}{329.5243}
\psarc[linecolor=black, linewidth=.1, dimen=outer](0.33,1.6035506){0.8}{250}{63}
\psarc[linecolor=black, linewidth=0.1, dimen=outer](7.63,1.3035506){0.8}{112.23932}{282.6804}
\psarc[linecolor=black, linewidth=0.1, dimen=outer](2.03,-3.3964493){0.8}{333}{149}
\psarc[linecolor=black, linewidth=0.1, dimen=outer](5.13,-3.6964495){0.8}{21.317913}{192}
\psarc[linecolor=black, linewidth=.1, dimen=outer](7.43,-1.6964494){0.8}{66.55701}{246.0375}
\psline[linecolor=black, linewidth=0.06, arrowsize=0.3cm 2.0,arrowlength=1.4,arrowinset=0.0]{->}(6.33,3.1035507)(7.33,2.0035505)
\psline[linecolor=black, linewidth=0.06, arrowsize=0.3cm 2.0,arrowlength=1.4,arrowinset=0.0]{->}(7.73,-0.89644945)(7.83,0.60355055)
\psline[linecolor=black, linewidth=0.06, arrowsize=0.3cm 2.0,arrowlength=1.4,arrowinset=0.0]{->}(5.83,-3.3964493)(7.13,-2.3964493)
\psline[linecolor=black, linewidth=0.06, arrowsize=0.3cm 2.0,arrowlength=1.4,arrowinset=0.0]{->}(4.33,-3.8964493)(2.73,-3.7964494)
\psline[linecolor=black, linewidth=0.06, arrowsize=0.3cm 2.0,arrowlength=1.4,arrowinset=0.0]{->}(0.53,-1.9964495)(1.33,-2.9964495)
\psline[linecolor=black, linewidth=0.06, arrowsize=0.3cm 2.0,arrowlength=1.4,arrowinset=0.0]{->}(0.03,0.90355057)(0.03,-0.39644945)
\psline[linecolor=black, linewidth=0.06, arrowsize=0.3cm 2.0,arrowlength=1.4,arrowinset=0.0]{->}(0.73,2.3035505)(1.73,3.3035505)
\psline[linecolor=black, linewidth=0.06, arrowsize=0.3cm 2.0,arrowlength=1.4,arrowinset=0.0]{->}(4.83,3.8035505)(3.13,3.9035506)
\end{pspicture}
}
\end{center}
Clearly, if $S$ has $n$ elements, $*_D\(S\)$ has $2^{n-1}(n-1)!$
elements. The structure operations glue the tips of the
cogs in such a way that the arrows go in the opposite directions, as in
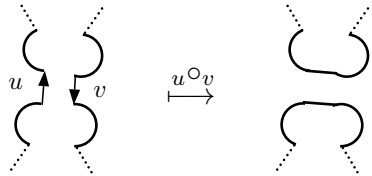
\begin{center}
\psscalebox{.35 .35} 
{
\begin{pspicture}(0,-3.1030016)(13.636998,3.1030016)
\psarc[linecolor=black, linewidth=0.1, dimen=outer](12.5057335,-1.400521){0.8}{267}{114}
\rput{-5.0}(-0.060016807,1.1135751)
{\psarc[linecolor=black, linewidth=.1, dimen=outer](12.722523,1.244093){0.8}{243}{90}}
\psarc[linecolor=black, linewidth=.1, dimen=outer](11.346681,1.4704981){0.8}{112.23932}{267.8539}
\rput{5.0}(-0.07798805,-0.9674769){\psarc[linecolor=black, linewidth=.1, dimen=outer](11.040439,-1.3768485){0.8}{78.5253}{246.0375}}
\psline[linecolor=black, linewidth=0.1](11.146681,0.6704981)(12.346681,0.57049805)
\psline[linecolor=black, linewidth=0.1](11.04668,-0.529502)(12.346681,-0.62950194)
\psline[linecolor=black, linewidth=0.1, linestyle=dotted, dotsep=0.10583334cm](12.846681,2.070498)(13.54668,3.070498)
\psline[linecolor=black, linewidth=0.1, linestyle=dotted, dotsep=0.10583334cm](13.146681,-3.229502)(12.54668,-2.229502)
\psline[linecolor=black, linewidth=0.1, linestyle=dotted, dotsep=0.10583334cm](10.04668,-3.129502)(10.74668,-2.129502)
\psline[linecolor=black, linewidth=0.1, linestyle=dotted, dotsep=0.10583334cm](11.04668,2.1704981)(10.44668,3.1704981)
\psarc[linecolor=black, linewidth=.1, dimen=outer](2.4057336,-1.400521){0.8}{267.07745}{95}
\rput{-5.0}(-0.098450355,0.23330204){\psarc[linecolor=black, linewidth=.1, dimen=outer](2.6225226,1.244093){0.8}{255.75034}{90}}
\psarc[linecolor=black, linewidth=0.1, dimen=outer](1.2466805,1.4704981){0.8}{112.23932}{267.8539}
\rput{5.0}(-0.1164216,-0.08720393){\psarc[linecolor=black, linewidth=0.1, dimen=outer](0.94043833,-1.3768485){0.8}{72.48264}{246.0375}}
\psline[linecolor=black, linewidth=0.06, arrowsize=0.3cm 2.0,arrowlength=1.4,arrowinset=0.0]{->}(1.1466805,-0.62950194)(1.2466805,0.6704981)
\psline[linecolor=black, linewidth=0.06, arrowsize=0.3cm 2.0,arrowlength=1.4,arrowinset=0.0]{->}(2.4086409,0.48962328)(2.3466804,-0.62950194)
\psline[linecolor=black, linewidth=0.1, linestyle=dotted, dotsep=0.10583334cm](2.7466805,2.070498)(3.4466805,3.070498)
\psline[linecolor=black, linewidth=0.1, linestyle=dotted, dotsep=0.10583334cm](3.0466805,-3.229502)(2.4466805,-2.229502)
\psline[linecolor=black, linewidth=0.1, linestyle=dotted, dotsep=0.10583334cm](-0.05331955,-3.129502)(0.6466805,-2.129502)
\psline[linecolor=black, linewidth=0.1, linestyle=dotted, dotsep=0.10583334cm](0.9466804,2.1704981)(0.34668046,3.1704981)
\rput(0.14668044,0.17049804){\psscalebox{2.5 2.5}{$u$}}
\rput(3.34604,-0.12950195){\psscalebox{2.5 2.5}{$v$}}
\rput(6.804,0.17049804){\psscalebox{3 3}{{$\stackrel{\raisebox{.2em}{$\ooo uv$}}{\longmapsto}$}}}
\end{pspicture}
}
\end{center}
The operad $*_D$ is the {M\"obiusisation\/} \cite[Definition
3.32]{braun} of the terminal cyclic operad $*_C$, the subscript $D$
referring to the dihedral structure \cite[Section~3]{markl-remm:JA06}
that $*_D$ carries.  Algebras over its linearization $\Span(*_D)$ are
associative algebras with involution \cite[Proposition~3.9]{braun}.
The Chuang-Lazarev approach applies also to this situation, except
that the sewing may not preserve the orientations now, see
Figure~\ref{fig:orandnonor}--right. Indeed, C.~Braun (who of course
worked in $\k-Mod$ not in $\Set$) proved

\begin{theorem}[{\cite[Theorem~3.10]{braun}}]
  The component \, $\Mod(*_D)\(S;g\)$ \, of the modular envelope $\Mod(*_D)$ is
  the set of isomorphism classes of (not necessarily oriented) surfaces
 with $b$ teethed holes whose teeth are labeled by $S$, with
  $m$ handles and $u$ crosscaps such that $g = 2m + b +u-1$.
\end{theorem}

As we theorized in the Introduction, we believe that $\Mod(*_D)$ is
(related to) the terminal operad in a suitable category of dihedral operads.

\appendix

\section{Cyclic and modular operads}
\label{appendix}

There are two  versions of biased definitions of
operads. The skeletal version has natural numbers as
the arities, in the non-skeletal the arities are finite
sets. We recall, following \cite[Section~2]{doubek:modass}, the
non-skeletal definitions of classical cyclic and modular operads.

\begin{definition}
A {\em cyclic module\/} is a functor
$
E : \Fin  \rightarrow \ttM
$
from the category of finite sets and their isomorphisms to 
our fixed symmetric monoidal category $\ttM$. 
\end{definition}

\begin{definition}
\label{cyclic_Jarka_nepise}
A {\em cyclic  operad\/} in $\ttM = (\ttM,\ot,1)$  is a cyclic module
\[
\oP = \big\{\oP\(S\) \in \ttM;\  S   \in  \Fin \big\}
\]
together with morphisms (compositions)
\begin{equation}
\label{eq:7_Jarka_nepise}
\ooo{u}{v}:\oP\(S'\)
\otimes \oP\(S''\)  \to \oP\bl S' \cup S'' \setminus \{u,v\}\br
\end{equation}
defined for arbitrary disjoint   
sets $S'$ and $S''$ with elements $u \in S'$,
$v \in S''$.
These data are required to satisfy the following axioms.
\begin{enumerate}
\itemindent -1em
 \itemsep .3em 
\item [(i)]
For $S'$, $S''$ and $u$, $v$ as in~(\ref{eq:7_Jarka_nepise}),
one has the equality
\[
\ooo{u}{v} =  \ooo{v}{u} \tau
\]
of maps $\oP\(S'\)\otimes \oP\(S''\) \to
\oP\bl S'\cup S'' 
\setminus  \{u,v\}\br$, where $\tau$ is  the symmetry
constraint in $\ttM$.

\item  [(ii)]
For mutually disjoint   sets
  $S_1,S_2,S_3$, 
and $a \in S_1$, $b,c \in
  S_2$, $b \not= c$, $d  \in S_3$, one has the equality
\[
\ooo ab (\id \ot \ooo cd)  = \ooo cd (\ooo ab \ot
\id)
\] 
of maps $\oP\(S_1\) \otred \oP\(S_2\) \otred  \oP\(S_3\) \to 
\oP\bl S_1 \cupred S_2 \cupred S_3 \setminus \{a,b,c,d\}\br$.

\item  [(iii)]
For arbitrary isomorphisms $\rho : S'\to D'$
  and $\sigma : S''\to D''$ of   sets and
  $u$, $v$ as in~(\ref{eq:7_Jarka_nepise}), one has the equality
\[
\oP\bl\rho|_{S'\setminus \{u\}}\cup\sigma|_{S''\setminus
  \{v\}}\br 
\ooo{u}{v} =
\ooo{\rho(u)}{\sigma(v)} \ \big(\oP\(\rho\)\ot\oP\(\sigma\)\big)
\]
of maps
$\oP\(S'\)\otimes \oP\(S''\) \to
\oP\bl D'\cup D'' 
\setminus  \{\rho(u),\sigma(v)\}\br$.
\end{enumerate} 
\end{definition}

The category $\Fin$ of finite sets is equivalent to its full skeletal
subcategory $\Fin_{\it sk}$ whose objects are the sets $[n] := \{1,\dots,n\}$, $n \geq
0$, with $[0]$ interpreted as the empty set~$\emptyset$. The
components of the skeletal
version of $\calP$ are
\[
\oP(n) := \oP\bl [n\!+\!1]\br, \ n \geq -1,
\]
with the induced action of the
symmetric group $\Sigma_{n+1} = {\rm Aut}\big([n\!+\!1]\big)$. The
structure operations
\[
\ooo ij:  \oP(m) \ot \oP(n) \to \oP(m\!+\!n\!-\! 1),\
\ 0 \leq i \leq m, \ \ 0 \leq j \leq n,
\] 
are induced from the equivalence of $\Fin$ with
$\Fin_{\it sk}$. 

Notice that we allow also the component $\oP\(\,\emptyset\,\) = \oP(-1)$ and the
operation
\[
\ooo uv : \oP \bl \stt u \br \ot  \oP \bl \stt v \br \to
\oP\(\,\emptyset\,\) \hbox { resp.} \ \ooo 00 : \calP(0) \ot \calP(0)
\to \calP(-1),
\]
while the original
definition~\cite[Theorem~2.2]{getzler-kapranov:CPLNGT95} always
requires an `output,' i.e.\ the arities must be non-empty sets (or $n
\geq 0$ in the skeletal $\oP(n)$). We do not demand operadic units.

\begin{definition}
A {\em modular module\/} is a functor
\[
E : \Fin \times \bbN  \rightarrow \ttM,
\]
where the natural numbers $\bbN = \{0,1,2,\ldots\}$ are considered as a
discrete category.
\end{definition}

\begin{definition}
\label{modular}
A { modular operad\/} in $\ttM = (\ttM,\ot,1)$  is a  modular module
\begin{equation}
\label{Jaruska_by_uz_mohla_napsat}
\oP = \big\{\oP\(S;g\) \in \ttM;\  (S,g)   \in  \Fin \times \bbN \big\}
\end{equation}
together with morphisms (compositions)
\[
\ooo{u}{v}:\oP\(S';g'\)
\otimes \oP\(S'';g''\)  \to \oP\bl S'\cup S'' \setminus \{u,v\};g'+g''\br
\]
defined for arbitrary disjoint   
sets $S'$ and $S''$ with elements $u \in S'$,
$v \in S''$, and contractions
\[
\xi_{uv} = \xi_{vu} : \oP\(S;g\) \to \oP\bl S \setminus \{u,v\};g+1\br
\]
given for any   
set $S$ and distinct elements $u,v \in S$.
These data are required to satisfy the following axioms.
\begin{enumerate}
\itemindent -1em
 \itemsep .3em 
\item [(i)]
For $S'$, $S''$ and $u$, $v$ as in~(\ref{Jaruska_by_uz_mohla_napsat}),
one has the equality
\[
\ooo{u}{v} =  \ooo{v}{u} \tau
\]
of maps $\oP\(S';g'\)\otimes \oP\(S'';g''\) \to
\oP\bl S'\cup S'' 
\setminus  \{u,v\};g'+g''\br$.

\item  [(ii)]
For mutually disjoint   sets
  $S_1,S_2,S_3$, 
and $a \in S_1$, $b,c \in
  S_2$, $b \not= c$, $d  \in S_3$, one has the equality
\[
\ooo ab (\id \ot \ooo cd)  = \ooo cd (\ooo ab \ot
\id)
\] 
of maps $\oP\(S_1;g_1\) \otred \oP\(S_2;g_2\) \otred  \oP\(S_3;g_3\) \to 
\oP\bl S_1 \cupred S_2 \cupred S_3 \setminus \{a,b,c,d\};
g_1\!+\!
g_2\! +\! g_3\br$.

\item  [(iii)]
For a   set $S$ and mutually distinct $a,b,c,d \in S$,
one has the equality
\[
\xi_{ab} \ \xi_{cd} = \xi_{cd} \ \xi_{ab}
\]
of maps $\oP\(S;g\) \to \oP\bl S \setminus \{a,b,c,d\};g+2\br$.

\item  [(iv)]
For   sets $S', S''$ and distinct $a,c \in
S'$, $b,d \in S''$, one has the equality
\[
\xi_{ab} \ \ooo{c}{d} = \xi_{cd} \ \ooo{a}{b}
\]
of maps $\oP\(S' \cup S'';g\) \to \oP\bl S' \cup S'' \setminus
\{a,b,c,d\};g+1\br$.

\item [(v)] For   sets $S', S''$ and
  mutually distinct $a,c,d \in S'$, $b \in S''$, one has the equality
\[
\ooo{a}{b} \ (\xi_{cd}\ot\id) = \xi_{cd} \ \ooo{a}{b}
\]
of maps $\oP\(S' \cup S'';g\) \to \oP\bl S' \cup S'' \setminus
\{a,b,c,d\};g+1\br$.

\item  [(vi)]
For arbitrary isomorphisms $\rho : S'\to D'$
  and $\sigma : S''\to D''$ of   sets and
  $u$, $v$ as in~(\ref{eq:7}), one has the equality
\[
\oP\bl\rho|_{S'\setminus \{u\}}\cup\sigma|_{S''\setminus
  \{v\}}\br 
\ooo{u}{v} =
\ooo{\rho(u)}{\sigma(v)} \ \big(\oP\(\rho\)\ot\oP\(\sigma\)\big)
\]
of maps
$\oP\(S';g'\)\otimes \oP\(S'';g''\) \to
\oP\bl D'\cup D'' 
\setminus  \{\rho(u),\sigma(v)\};g'+g''\br$.

\item  [(vii)]
For $S$, $u$, $v$ as in~(\ref{eq:8}) and an isomorphism $\rho :
S \to D$ of   sets, one has the equality
\[
\oP\bl\rho|_{D \setminus \{\rho(u),\rho(v)\}}\br \ \xi_{ab} = 
\xi_{\rho(u)\rho(v)}\oP\(\rho\)
\]
of maps
$\oP\(S;g\) \to \oP\bl S \setminus \{\rho(u),\rho(v)\};g+1\br$.
\end{enumerate} 
\end{definition}

Informally, cyclic operads are modular operads without the
contractions and the genus grading. In the seminal
paper \cite{getzler-kapranov:CompM98} where modular operads were
introduced, the {\em stability\/} demanding that
\[
\oP\(S;g\) = 0 \ \hbox { if $\card(S) < 3$ and $g=0$, or $\card(S) =0$ and $g=1$}
\] 
was assumed, but we do not require this property.  As a matter of
fact, our main examples of \ns\ modular operads are not stable, though
stable versions of of our results can easily be formulated and proved.

\bibliography{b}
\def\cprime{$'$}\def\cprime{$'$}

\end{document}